\documentclass[a4paper,11pt,final]{amsart}

\usepackage{anysize} \marginsize{1.3in}{1.3in}{1in}{1in}
\usepackage{xcolor}
\usepackage{amsmath}
\usepackage{mathtools}
\usepackage[all]{xy}
\usepackage[latin1]{inputenc}
\usepackage{varioref}
\usepackage{amsfonts}
\usepackage{amssymb}
\usepackage{bbm}
\usepackage{graphicx}
\usepackage{mathrsfs}
\usepackage[hypertexnames=false,backref=page,pdftex,
 	pdfpagemode=UseNone,
 	breaklinks=true,
 	extension=pdf,
 	colorlinks=true,
 	linkcolor=blue,
 	citecolor=red,
 	urlcolor=blue,
 ]{hyperref}

\newcommand{\sO}{{\mathcal O}}

\newcommand{\sU}{{\mathcal U}}

\newcommand{\sX}{{\mathcal X}}

\newcommand{\scrB}{{\mathscr B}}
\newcommand{\scrC}{{\mathscr C}}

\newcommand{\scrE}{{\mathscr E}}
\newcommand{\scrF}{{\mathscr F}}

\newcommand{\scrH}{{\mathscr H}}

\newcommand{\scrP}{{\mathscr P}}

\newcommand{\scrS}{{\mathscr S}}

\newcommand{\scrX}{{\mathscr X}}
\newcommand{\scrY}{{\mathscr Y}}
\newcommand{\scrZ}{{\mathscr Z}}

\newcommand{\C}{{\mathbb C}}

\newcommand{\HH}{{\mathbb H}}

\newcommand{\N}{{\mathbb N}}
\renewcommand{\P}{{\mathbb P}}
\newcommand{\Q}{{\mathbb Q}}
\newcommand{\R}{{\mathbb R}}

\newcommand{\Z}{{\mathbb Z}}

\newcommand{\Br}{{\operatorname{Br}}}

\newcommand{\id}{{\rm id}}

\newcommand{\chern}{{\rm c}}
\newcommand{\codim}{\operatorname{codim}}

\newcommand{\coker}{\operatorname{coker}}

\newcommand{\Ext}{\operatorname{Ext}}

\newcommand{\Hilb}{{\rm Hilb}}

\renewcommand{\Im}{\operatorname{Im}}

\newcommand{\into}{{\, \hookrightarrow\,}}
\newcommand{\isom}{{\ \cong\ }}

\newcommand{\lt}{{\rm{lt}}}

\renewcommand{\O}{{\rm O}}

\newcommand{\ol}[1]{{\overline{#1}}}
\newcommand{\onto}{{{\twoheadrightarrow}}}

\newcommand{\ot}[1][]{\xleftarrow{\ #1\ }}

\newcommand{\Pic}{\operatorname{Pic}}

\renewcommand{\Re}{\operatorname{Re}}
\newcommand{\ratl}{\dashrightarrow}

\newcommand{\reg}{{\operatorname{reg}}}

\newcommand{\rk}{{\rm rk}}
\newcommand{\sing}{{\operatorname{sing}}}

\newcommand{\Stab}{{\operatorname{Stab}}}
\newcommand{\st}{{\operatorname{1st}}}

\newcommand{\Sym}{{\rm Sym}}
\renewcommand{\to}[1][]{\xrightarrow{\ #1\ }}

\newcommand{\tensor}{\otimes}
\newcommand{\tr}{{\rm tr}}

\newcommand{\vphi}{\varphi}

\newcommand{\ul}[1]{{\underline{#1}}}

\renewcommand{\st}{{\rm st}}
\newcommand{\hyplat}{{\mathcal{H}}}

\newcommand{\gothm}{{\mathfrak{m}}}

\newtheoremstyle{citing}
  {}
  {}
  {\itshape}
  {}
  {\bfseries}
  {\textbf{.}}
  {.5em}
  {\thmnote{#3}}

\theoremstyle{plain}

\newtheorem{theorem}[subsection]{Theorem}
\newtheorem{lemma}[subsection]{Lemma}
\newtheorem{proposition}[subsection]{Proposition}

\newtheorem{corollary}[subsection]{Corollary}
\theoremstyle{definition}

\newtheorem{definition}[subsection]{Definition}

\newtheorem{Ex}[subsection]{Example}

\numberwithin{equation}{section}

\theoremstyle{remark}
\newtheorem{setup}[subsection]{Setup}
\newtheorem{remark}[subsection]{Remark}

{\theoremstyle{citing}
}

\newcommand{\gothN}{{\mathfrak{N}}}
\newcommand{\gothM}{{\mathfrak{M}}}

\newcommand{\Hdg}{\operatorname{Hdg}}

\newcommand{\cone}{{\operatorname{cone}}}

\newcommand{\Exz}{{\operatorname{Exc}}}

\newcommand{\tX}{{Y}}

\newcommand{\Def}{{\operatorname{Def}}}

\newcommand{\deflt}[1][X]{{\Def^\lt(#1)}}
\newcommand{\Ch}{{\textrm{Ch}}}
\newcommand{\Kthree}{$K3\;$}
\newcommand{\Kthreen}{$K3^{[n]}$}

\renewcommand{\tilde}{\widetilde}
\def\SO{\operatorname{SO}}

\def\Mon{\operatorname{Mon}}

\newcommand{\res}{\mathrm{res}}
\renewcommand{\div}{\operatorname{div}}

\def\rrk{\operatorname{rrk}}

\theoremstyle{plain}
\newtheorem{question}[subsection]{Question}
\theoremstyle{remark}
\newtheorem{example}[subsection]{Example}

\makeatletter
\newcommand*{\da@rightarrow}{\mathchar"0\hexnumber@\symAMSa 4B }
\newcommand*{\da@leftarrow}{\mathchar"0\hexnumber@\symAMSa 4C }
\newcommand*{\xdashrightarrow}[2][]{%
  \mathrel{%
    \mathpalette{\da@xarrow{#1}{#2}{}\da@rightarrow{\,}{}}{}%
  }%
}
\newcommand{\xdashleftarrow}[2][]{%
  \mathrel{%
    \mathpalette{\da@xarrow{#1}{#2}\da@leftarrow{}{}{\,}}{}%
  }%
}
\newcommand*{\da@xarrow}[7]{%
  \sbox0{$\ifx#7\scriptstyle\scriptscriptstyle\else\scriptstyle\fi#5#1#6\m@th$}%
  \sbox2{$\ifx#7\scriptstyle\scriptscriptstyle\else\scriptstyle\fi#5#2#6\m@th$}%
  \sbox4{$#7\dabar@\m@th$}%
  \dimen@=\wd0 %
  \ifdim\wd2 >\dimen@
    \dimen@=\wd2 %
  \fi
  \count@=2 %
  \def\da@bars{\dabar@\dabar@}%
  \@whiledim\count@\wd4<\dimen@\do{%
    \advance\count@\@ne
    \expandafter\def\expandafter\da@bars\expandafter{%
      \da@bars
      \dabar@ 
    }%
  }%
  \mathrel{#3}%
  \mathrel{%
    \mathop{\da@bars}\limits
    \ifx\\#1\\%
    \else
      _{\copy0}%
    \fi
    \ifx\\#2\\%
    \else
      ^{\copy2}%
    \fi
  }%
  \mathrel{#4}%
}
\makeatother

\title[Global Torelli for singular symplectic varieties]{A global Torelli theorem for singular symplectic varieties}

\author{Benjamin Bakker}
\address{Benjamin Bakker\\ Department of Mathematics\\ University of Illinois at Chicago\\ 851 S. Morgan St., Chicago,
IL 60607}
\email{bakker.uic@gmail.com}
\author{Christian Lehn}
\address{Christian Lehn\\ Fakult\"at f\"ur Mathematik\\ Technische Universit\"at Chemnitz\\
Reichenhainer Stra\ss e 39, 09126 Chemnitz, Germany}
\email{christian.lehn@mathematik.tu-chemnitz.de}

\let\origmaketitle\maketitle
\def\maketitle{
  \begingroup
  \def\uppercasenonmath##1{} 
  \let\MakeUppercase\relax 
  \origmaketitle
  \endgroup
}

\begin{document}
\thispagestyle{empty}

\begin{abstract}
We systematically study the moduli theory of symplectic varieties (in the sense of Beauville) which admit a resolution by an irreducible symplectic manifold.  In particular, we prove an analog of Verbitsky's global Torelli theorem for the locally trivial deformations of such varieties.  Verbitsky's work on ergodic complex structures replaces twistor lines as the essential global input.  In so doing we extend many of the local deformation-theoretic results known in the smooth case to such (not-necessarily-projective) symplectic varieties.  We deduce a number of applications to the birational geometry of symplectic manifolds, including some results on the classification of birational contractions of \Kthreen-type varieties.  
\end{abstract}

\maketitle

\setlength{\parindent}{1em}
\setcounter{tocdepth}{1}


\tableofcontents

\section{Introduction}\label{sec intro}
\thispagestyle{empty}

The local and global deformation theories of irreducible holomorphic symplectic manifolds enjoy many beautiful properties. For example, unobstructedness of deformations \cite{Bo78,Ti,To89} and the local Torelli theorem \cite{B} are at the origin of many results on symplectic manifolds. Among the highlights of the global theory are Huybrechts' surjectivity of the period map \cite{Huy99} and Verbitsky's global Torelli theorem \cite{V13} (see for example Markman's survey article \cite{markmantor} and Huybrechts' Bourbaki talk \cite{huybourbaki}). 
Verbitsky's result has since paved the way for many important developments, with applications to a wide variety of questons such as:  birational boundedness \cite{Ch16}, lattice polarized mirror symmetry \cite{Ca16}, algebraic cycles \cite{CP}, hyperbolicity questions \cite{KLV} and many more.  Recent progress in MMP and interest in singular symplectic varieties \cite{GKP} has made it apparent that a global deformation theory of singular symplectic varieties would be equally valuable.  As for smooth varieties, it is of utmost importance here to consider not-necessarily-projective K\"ahler varieties---even if one is mainly interested in moduli spaces of projective varieties, see e.g. Remark \ref{remark ghs}.

In the present article we initiate a systematic study of the moduli theory of compact K\"ahler singular symplectic varieties, beginning with those that admit symplectic resolutions by irreducible symplectic manifolds. 
We prove general results concerning their deformation theory and building on this we develop a global moduli theory for locally trivial families of such varieties. 
To fix ideas, recall that a \emph{symplectic variety} $X$ in the sense of Beauville \cite[Definition 1.1]{Be} is a normal complex variety whose regular part admits a nondegenerate holomorphic 2-form that extends to some (hence any) resolution.  Usually the extension will have zeroes, but if there is a resolution $\pi:Y\to X$ by an (irreducible) holomorphic symplectic manifold $Y$, we call $\pi$ an \emph{(irreducible) symplectic resolution}.

Namikawa shows \cite[Theorem 2.2]{Na01} that a projective variety $X$ admitting a resolution $\pi:Y\to X$ by a symplectic manifold $Y$ has unobstructed deformations, and that there is a natural finite map $\Def(Y)\to\Def(X)$.  However, a generic deformation of $\pi$ becomes an isomorphism, and therefore the only natural period map is that of the resolution $Y$.
From a Hodge-theoretic perspective it is therefore more natural to consider the \emph{locally trivial} deformations of $X$ as in this case the pure weight two Hodge structure on $H^2(X,\Z)$ varies in a local system, and the resulting theory is very closely analogous to the smooth situation. 
Our first result is the following

\begin{theorem}[see Propositions \ref{prop defo} and \ref{proposition nonproj deform}]\label{theorem def smooth} Let $\pi:Y\to X$ be an irreducible symplectic resolution of a symplectic variety $X$ and $N=N_1(Y/X)\subset H_2(Y,\Z)$ the group of 1-cycles contracted by $\pi$.  The base space $\Def^\lt(X)$ of the universal locally trivial deformation  is smooth, and there is a diagram
\[
\xymatrix{
\scrY \ar[d]\ar[r]& \scrX \ar[d]\\
\Def(Y,N) \ar[r]^{\cong} & \Def^\lt(X) \\
}
\]
where $\scrX\to \Def^\lt(X)$ is the universal locally trivial deformation of $X$, $\scrY\to\Def(Y,N)\subset\Def(Y)$ is the restriction of the universal deformation of $Y$ to the closed subspace along which $N$ remains Hodge, and $\scrY\to\scrX$ specializes to $\pi$.
\end{theorem}
 The case of a divisorial contraction $\pi:Y\to X$ of a projective symplectic variety was treated by Pacienza and the second author in \cite[Proposition 2.3]{LP}, where it was shown that locally trivial deformations of $X$ correspond to deformations of $Y$ such that all irreducible components of the exceptional divisor $\Exz(\pi)$ of $\pi$ deform along. The space $\Def^\lt(X)$ of locally trivial deformations of $X$ is smooth of dimension $h^{1,1}(X)=h^{1,1}(Y)-m$ where $m$ is the number of irreducible components of $\Exz(\pi)$.  This description is equivalent to that of Theorem \ref{theorem def smooth}, since the Beauville--Bogomolov--Fujiki form $q_Y$ on $H^2(Y,\Z)$ yields an isomorphism $\tilde q_Y: H^2(Y,\Q)\xrightarrow{\cong} H_2(Y,\Q)$ identifying the subspace of $H^2(Y,\Q)$ spanned by $\Exz(\pi)$ with $N_\Q$.  In the case where $\pi$ is a small contraction, $X$ is not $\Q$-factorial, but we can still recast the description in Theorem \ref{theorem def smooth} in terms of line bundles:  under $\tilde q_Y$, the Hodge classes in the orthogonal $N^\perp\subset H^2(Y,\Q)$ are the $\Q$-line bundles on $Y$ that vanish on $N$ and can be pushed forward to $\Q$-line bundles on $X$ (as the singularities are rational).

An important theorem of Huybrechts \cite[Theorem 2.5]{Huy} shows that birational \footnote{We will use the term birational instead of bimeromorphic for compact K\"ahler varieties as well.}
irreducible holomorphic symplectic manifolds are deformation equivalent.  Of course, birational singular symplectic varieties are not necessarily locally trivially deformation equivalent, and the correct analog of Huybrechts' theorem is the following. 

\begin{theorem}[See Theorem \ref{theorem huybrechts}]\label{intro cor div}
Let $\pi:Y \to X$ and $\pi':Y' \to X'$ be irreducible symplectic resolutions of symplectic varieties $X$, $X'$. Assume that there is a birational map $\phi: Y \ratl Y'$ such that the induced map $\phi^*: H^2(\tX',\C) \to H^2(\tX,\C)$ sends $H^2(X',\C)$ isomorphically to $H^2(X,\C)$.  Then $X$ and $X'$ are locally trivial deformations of one another.
\end{theorem}

Let us now turn to global moduli theoretic questions. Fixing a lattice $\Lambda$, there is a natural locally trivial $\Lambda$-marked moduli space $\mathfrak{M}^\lt_\Lambda$ obtained by gluing the universal locally trivial deformation spaces together, and we also have a corresponding notion of parallel transport operator.  Further, there is a period map $P:\mathfrak{M}^\lt_\Lambda\to\Omega_\Lambda$ to the associated period domain $\Omega_\Lambda$ that is a local isomorphism.  

The following is an analog of Verbitsky's global Torelli theorem \cite[Theorem 1.17]{V13}, also encompassing analogs of Huybrechts' surjectivity of the period map \cite[Theorem 8.1]{Huy99}, and Markman's description of the fibers of the period map in terms of the K\"ahler cone decomposition \cite[Theorem 5.16]{markmantor}. 

\begin{theorem}\label{intro theorem torelli}Let $X$ be a symplectic variety with $b_2(X)>4$ {that admits an irreducible symplectic resolution}.  Let $\mathfrak{N}^\lt_\Lambda\subset \mathfrak{M}^\lt_\Lambda$ be a connected component containing $(X,\nu)$.  Then
\begin{enumerate}
\item The period map $P:\mathfrak{N}^\lt_\Lambda\to\Omega_\Lambda$ is surjective, {generically injective,} and the points in any fiber are pairwise nonseparated.  {Moreover, varieties underlying points in the same fiber are birational.}
\item The points in the fiber containing $(X,\nu)$ are in bijective correspondence with the cones obtained by restricting the K\"ahler chambers of a resolution to $H^{1,1}(X,\R)$.
\item   The locally trivial weight two monodromy group $\Mon^2(X)^\lt$ is a finite index subgroup of $\O(H^2(X,\Z))$. 
\end{enumerate}
\end{theorem}

See Section \ref{sec monodromy} for more precise statements. The finiteness of the index of the monodromy group in the smooth case is proved by Verbitsky \cite[Theorem 2.6]{Ve19}.

The analogous results in the smooth case heavily rely on the existence of a hyperk\"ahler metric, the theory of twistor lines, and deformations of complex structures. This presents a major difficulty for singular varieties as the aforementioned techniques are not available or much less understood as in the smooth case. 
We therefore deduce the largeness of the image of the period map by density results built on Ratner theory (as first explored in this context by Verbitsky in \cite{V15}, see e.g. Theorem 3.9 and Theorem 4.8 of that paper and Proposition \ref{proposition verbitsky} in Section 5 of the present paper), and this is responsible for the numerical conditions on $b_2(X)$.  To the best of our knowledge this is the first general result in this direction which makes a statement about \emph{large} deformations of singular symplectic varieties.  Note that $X$ with $b_2(X)=3$ are locally trivially rigid in the sense that their only locally trivial deformation is what should be the (unique) twistor deformation, and unlike in the smooth case we can give examples (see Remark \ref{remark rigid}). 
These density results have many applications, and for instance we have the following application pointed out by Amerik--Verbitsky \cite[Theorem 5.6]{AV19}.
\begin{theorem}[See Corollary \ref{corollary can contract}]\label{theorem can contract}Let $Y$ be a smooth irreducible symplectic manifold with $b_2(Y)>5$ and $\tau\subset H^{1,1}(Y,\R)$ a face\footnote{By a face of an open cone we mean a linear subspace meeting the closure with nonempty interior.} of the K\"ahler cone meeting the positive cone for which\footnote{By this we mean the dimension of the real linear subspace generated by $\tau$.  Note that in this case we know the K\"ahler cone is locally polyhedral in the positive cone---see Remark \ref{remark MBM gen}.} $\dim\tau>2$.  Then there is a birational contraction $\pi:Y\to X$ contracting precisely $\tau^\perp$.
\end{theorem}
Somewhat surprisingly, we also obtain a generalization to the non-projective setting of a result of Namikawa, see \cite[Theorem (2.2)]{Na01}.

  \begin{theorem}[See Proposition \ref{prop namikawa general}]\label{thm namikawa general}Let $\pi:Y\to X$ be an irreducible symplectic resolution with $b_2(X)>4$.  Then $\Def(X)$ is smooth of the same dimension as $\Def(Y)$ and the induced map $p:\Def(Y)\to\Def(X)$ is finite.
  \end{theorem}
This is noteworthy because one important ingredient in Namikawa's proof of the corresponding statement for projective varieties is a vanishing theorem by Steenbrink  which we do not know to hold in the K\"ahler case. Our proof works by reduction to the projective case using the monodromy action.

Much more can be said when $X$ admits a resolution deformation equivalent to a Hilbert scheme of points on a $K3$ surface.  The possible projective contractions $M\to X$ of a moduli space of sheaves $M$ on a $K3$ surface is completely described by wall-crossing in the space of Bridgeland stability conditions by work of Bayer--Macr\`i \cite{BM}; 
Theorem \ref{theorem def smooth} combined with density results then implies that no new singularities occur for general $X$:

\begin{theorem}[See Proposition \ref{proposition def to bridge}]Let $\pi:Y\to X$ be an irreducible symplectic resolution where $Y$ is a \Kthreen-type manifold.  Provided $b_2(X)>4$, $X$ is locally trivially deformation equivalent to a wall-crossing contraction of a moduli space of Bridgeland stable objects on a K3 surface.  
\end{theorem}

Extremal contractions are particularly amenable to the lattice theory involved, and we for example have the following generalization of a beautiful result of Arbarello and Sacc\`a \cite[Theorem 1.1 part i)]{AS}:

\begin{theorem}[See Proposition \ref{nakajima}]\label{intro theorem nakajima} Let $\pi:Y\to X$ be an irreducible symplectic resolution where $Y$ is a \Kthreen-type manifold, and assume $\rk\,N_1(Y/X)=1$.  Then for any closed point $x\in X$ the analytic germ $(X,x)$ is isomorphic to that of a Nakajima quiver variety.
\end{theorem}

As another nice application, we are able to rule out the existence of certain contractions on \Kthreen-type varieties:

\begin{theorem}[See Corollary \ref{corollary a2}] Let $\pi:Y\to X$ be an irreducible symplectic resolution where $Y$ is a \Kthreen-type manifold, and assume $\rk\,N_1(Y/X)=1$. Then for a generic closed point $x\in X^\sing$ the analytic germ $(X,x)$ is isomorphic to $(S,0)\times (\C^{2n-2},0)$ where $(S,0)$ is an $A_1$-surface singularity.
\end{theorem}

The surprising fact is that relative Picard rank one contractions whose generic singularity is transversally an $A_2$-surface singularity \emph{do not occur} on irreducible symplectic manifolds while on other symplectic manifolds they do, see Corollary \ref{corollary a2} and the discussion thereafter.

We expect most of our general results to apply to symplectic varieties not necessarily admitting a symplectic resolution by using a $\Q$-factorial terminalization in place of the resolution but we do not pursue that level of generality here.  We do note however that this would fit very nicely with another result of Namikawa \cite[Main Theorem, p.~97]{Na06} that every flat deformation of a $\Q$-factorial terminal projective symplectic variety is locally trivial.  As for the applications, we only restrict to \Kthreen-type varieties for simplicity, and our theory should yield similar results for the Kummer and O'Grady types.  These directions are the topic of a forthcoming article.

\subsection*{Outline}
In Sections \ref{sec hodge} and \ref{sec symplektisch} we explain some basic facts about symplectic varieties admitting irreducible symplectic resolutions and their Hodge structures. We recall that the Hodge structure on the second cohomology of such a symplectic variety is always pure and prove the degeneration of Hodge-de Rham spectral sequence for singular symplectic varieties in a suitable range. 

The locally trivial deformation theory of symplectic varieties admitting an irreducible symplectic resolution is studied in Section \ref{sec defo}. This is one of the two technical centerpieces of the present paper. All our moduli theory builds on this in an essential way. The main results are the proof of smoothness of the Kuranishi space, the comparison of deformations of a symplectic variety admitting an irreducible symplectic resolution and its resolution, the local Torelli theorem, and the analog of Huybrechts' theorem. 
Even though the deformations of a birational contraction can be characterized locally in the deformation space in terms of Noether-Lefschetz loci, it is---contrary to the surface case---nontrivial to show that a contraction deforms.  We also include some remarks and applications to the question of existence of algebraically coisotropic subvarieties. 

In Section \ref{sec monodromy} we introduce and investigate marked moduli spaces of locally trivial families of symplectic varieties admitting an irreducible symplectic resolution. We study the relation to moduli spaces of the resolution through compatibly marked moduli spaces of irreducible symplectic resolutions as well as various associated monodromy groups. The analysis of how these moduli spaces and their monodromy groups are connected is the second main ingredient of this work, and our analysis makes use of Verbitsky's ergodicity of complex structures (see \cite[Definition 1.12]{V15}) and Amerik-Verbitsky's concept of MBM classes (see \cite[Definition 1.13]{AV}). It becomes clear at this point that it is not sufficient to study the moduli spaces of lattice polarized varieties as this perspective ignores the question of how the contraction deforms on two fronts:  locally in terms of how the singularities smooth (and whether the contraction deforms at all) and globally in terms of how the monodromy group of the contraction differs from the isotropy group of the lattice polarization.

In Section \ref{sec k3type} we give applications to \Kthreen-type varieties. We essentially make use of the description of the birational geometry of Bridgeland moduli spaces of stable objects by Bayer and Macr\`i. In principle, our methods allow to extend any result on singularities that can be proven for contractions of Bridgeland moduli spaces to arbitrary \Kthreen-type varieties. The above-mentioned generalization of Arbarello--Sacc\`a's result is proven here.

\subsection*{Acknowledgments.} We would like to thank J\'anos Koll\'ar for helpful discussions on birational contractions. We benefited from discussions, remarks, emails of Chiara Camere, Daniel Greb, Klaus Hulek, Manfred Lehn, Giovanni Mongardi, Gianluca Pacienza, Jacob Tsimerman, Thomas Peternell, Stefan Kebekus, Stefano Urbinati, and Misha Verbitsky. Both authors are grateful to the referees for a very careful reading and many suggestions that have greatly improved the exposition.

Benjamin Bakker was supported by NSF grants DMS-1702149 and DMS-1848049. Christian Lehn was supported by the DFG through the research grants \mbox{Le 3093/2-1}, Le 3093/2-2, and  Le 3093/3-1. 

\enlargethispage{\baselineskip}

\subsection*{Notation and Conventions} 
The term variety will denote an integral separated scheme of finite type over $\C$ in the algebraic setting or an irreducible and reduced Hausdorff complex space in the complex analytic setting. We will also refer to either (complex) algebraic varieties or (complex) analytic varieties. Mostly, the term variety stands for an analytic variety though.\\
For a topological group $G$ we denote by $G^\circ$ the connected component of the identity.

\section{Hodge theory of rational singularities}\label{sec hodge}
In this section we establish some basic facts about the Hodge structure on low degree cohomology groups of varieties with rational singularities. Recall that the Fujiki class $\scrC$ consists of all those compact complex varieties which are meromorphically dominated by a compact K\"ahler manifold, see \cite[\S 1]{Fuj78}. This is equivalent to saying that there is a resolution of singularities by a compact K\"ahler manifold by Lemma 1.1 of op. cit. We will speak of a variety of Fujiki class for a variety in class $\scrC$. Let us also recall that by \cite[Corollary 1.7]{Fuj78} for each $k\geq 0$ the singular cohomology of degree $k$ of a smooth compact variety of Fujiki class carries a pure Hodge structure of weight $k$.

Let $X$ be a variety carrying a pure Hodge structure on $H^{2k}(X,\Z)$. We will write $H^{k,k}(X,\Z)$ for the preimage of $H^{k,k}(X)$ under the map $H^{2k}(X,\Z)\to H^{2k}(X,\C)$. Recall that for $k=1$ the \emph{transcendental lattice} $H^2(X)_\tr\subset H^2(X,\Z)$ is defined to be the smallest integral Hodge substructure such that $H^2(X,\C)_\tr:=H^2(X)_\tr\tensor \C$ contains $H^{2,0}(X)$.

While the following lemma is well-known, we include it for the reader's convenience.

\begin{lemma}\label{lemma peternell}
Let $\pi:Y\to X$ be a proper birational morphism where $X$ is a complex variety with rational singularities. Then, $\pi^*:H^1(X,\Z) \to H^1(Y,\Z)$ is an isomorphism and the sequence
\begin{equation}\label{eq leray}
 0 \to H^2(X,\Z) \to[\pi^*] H^2(Y,\Z) \to H^0(X,R^2\pi_*\Z)
\end{equation}
is exact. In particular, if $X$ is compact and $Y$ is a compact manifold of Fujiki class, then $H^i(X,\Z)$ carries a pure Hodge structure for $i=1, 2$. Moreover, the restriction $\pi^*: H^2(X)_\tr \to H^2(Y)_\tr$ is an isomorphism, and $\pi^*H^{1,1}(X,\Z)$ is the subspace of $H^{1,1}(Y,\Z)$ of all classes that vanish on the classes of $\pi$-exceptional curves.
\end{lemma}
\begin{proof}
The pushforward of the exponential sequence gives an exact sequence
$$\sO_X \to[\exp] \sO_X^\times  \to R^{1}\pi_*\Z_Y \to R^{1}\pi_*\sO_Y \to \ldots$$
where the map $\exp$ is again surjective. Thus, rationality of the singularities implies that $R^1\pi_*\Z=0$. The Leray spectral sequence tells us now that $\pi^*$ is an isomorphism on $H^1$ and that the sequence \eqref{eq leray} is exact. 
As $H^i(Y,\Z)$ carries a pure Hodge structure, also $H^i(X,\Z)$ carries a pure Hodge structure for $i=1,2$ by strictness of morphisms of Hodge structures.

Let us address the last two statements. By \cite[(12.1.3) Theorem]{KM}, the map $\pi_*:H_2(Y,\C)\to H_2(X,\C)$ is surjective and its kernel is generated by algebraic cycles, since $R^2\pi_*\sO_Y=0$ by the rationality of the singularities. Taking duals, we see that the restriction $\pi^*: H^2(X)_\tr \to H^2(Y)_\tr$ is an isomorphism. Moreover, from loc. cit. we also infer that $\ker \pi_*\subset H_2(Y,\C)$ consists of push forwards of homology classes contained in the fibers of $\pi$. A closer look at the proof reveals that it is in fact generated by algebraic cycles in the fibers. This is deduced from \cite[(12.1.1) Lemma]{KM} where it is shown using $R^2\pi_*\sO_Y=0$ that all degree $2$ homology in the fibers of $\pi$ is algebraic. Consequently, the kernel is generated by contracted curves and by taking duals again we may conclude the proof.
\end{proof}

For projective $X$ the purity of the Hodge structure on the second cohomology was observed by Namikawa, see e.g. the proof of \cite[Proposition 1]{Na06}. See also \cite[Theorem 7]{schwald} and \cite[Lemma 3.1]{PR}.

We will need to know to what extent the de Rham complex on either a resolution or the smooth part of a singular variety can be used to compute the singular cohomology and the Hodge decomposition on $X$. Note that there is an obvious morphism of complexes $\C_X\to \Omega_X^\bullet$, and for  any resolution $\pi:Y\to X$ there is a natural pullback map of complexes $\Omega_X^\bullet\to \pi_*\Omega_Y^\bullet$.  Likewise for an open embedding $j:U\to X$ whose complement has codimension $\geq 2$, there is a pullback $\Omega_X^\bullet\to j_*\Omega_U^\bullet$. The compositions $\C_X\to \Omega_X^\bullet \to j_*\Omega_U^\bullet$ and $\C_X\to \Omega_X^\bullet \to \pi_*\Omega_Y^\bullet$ induce canonical maps $H^k(X,\C)\to \HH^k(X,j_*\Omega_U^\bullet)$ and $H^k(X,\C)\to \HH^k(X,\pi_*\Omega_Y^\bullet)$ which we analyze in the following lemma.

\begin{lemma}\label{lemma hodge}
Let $X$ be a normal compact variety of Fujiki class with rational singularities and let us denote by $j:U=X^\reg\to X$ the inclusion of the regular part. Then the following hold:
\begin{enumerate}
\item \label{lemma hodge item zero} For each $p \in \N_0$ the sheaf $j_*\Omega_U^p$ is a coherent and reflexive $\sO_X$-module.
\item\label{lemma hodge item two} For all $k\leq 2$ the canonical map $H^k(X,\C)\to \HH^k(X,j_*\Omega_U^\bullet)$ is an isomorphism and the Hodge-de Rham spectral sequence
\begin{equation}\label{eq spec seq projective}
 E^{p,q}_1=H^q(j_*\Omega_U^p) \Rightarrow \HH^{p+q}(X,j_*\Omega_U^\bullet)
\end{equation}
degenerates on $E_1$ in the region where $p+q \leq 2$.
	\item\label{lemma hodge item one} Let $\pi:Y \to X$ be a resolution of singularities. Then for all $k\leq 2$ the canonical map $H^k(X,\C)\to \HH^k(X,\pi_*\Omega_Y^\bullet)$ is an isomorphism and the Hodge-de Rham spectral sequence
\begin{equation}\label{eq spec seq}
 E^{p,q}_1=H^q(\pi_*\Omega_Y^p) \Rightarrow \HH^{p+q}(X,\pi_*\Omega_Y^\bullet)
\end{equation}
degenerates on $E_1$ in the region where $p+q \leq 2$.
\end{enumerate}
\end{lemma}
\begin{proof}
Observe that $\Omega_U^p$ is locally free for each $p$, in particular torsion free. As $X$ is normal, we infer that $j_*\Omega_U^p$ is coherent by Serre's result \cite[Th\'eor\`eme 1]{Se66}. Here we use that $\Omega_X$ is a coherent extension of $\Omega_U$.

By \cite[Corollary 1.8]{KS18}, we have $j_*\Omega_U^\bullet=\pi_*\Omega_\tX^\bullet$ so that in particular $j_*\Omega_U^p$ is reflexive for every $p$ and \eqref{lemma hodge item zero} follows. This equality together with \eqref{lemma hodge item one} also implies \eqref{lemma hodge item two}, so it suffices to prove the third statement. For every morphism $C_\bullet\to D_\bullet$ of complexes on $Y$ the diagram
\begin{equation}\label{eq hodge pre diag}
\xymatrix{
\pi_*C_\bullet  \ar[r]\ar[d]  & \pi_*D_\bullet \ar[d]\\
R\pi_*C_\bullet  \ar[r] & R\pi_*D_\bullet \\
}
\end{equation}
of complexes on $X$ commutes. For $C_\bullet = \C \to \Omega_Y^\bullet=D_\bullet$ we obtain the commuting diagram
\begin{equation}\label{eq hodge diag}
\xymatrix{
H^k(X,\C) \ar[r]\ar[d]^{\pi^*} & \HH^k(\pi_*\Omega_Y^\bullet)\ar[d]^{\psi}\\
H^k(\tX,\C) \ar[r] & \HH^k(\Omega_\tX^\bullet)\\
}
\end{equation}
The lower horizontal map is an isomorphism by Grothendieck's theorem and $\pi^*$ is injective for $k\leq 2$ by Lemma \ref{lemma peternell}. 
 We will show that for $k\leq 2$ the map $\psi$ is injective and the codimension of its image is the same as that of $\pi^*$. 

For the injectivity we compare the spectral sequences on $X$ and $Y$. Let us show first that the $E_1$-level of the spectral sequence \eqref{eq spec seq} embeds into the $E_1$-level of the spectral sequence of the complex $\Omega_\tX^\bullet$, which degenerates on $E_1$ by Hodge theory. Note that $Y$ is also of Fujiki class and therefore it carries a Hodge structure on its cohomology and the Hodge-de Rham spectral sequence degenerates on $E_1$.

Now, we have $H^k(X,\sO_X)\isom H^k(\tX,\sO_\tX)$ by rationality of singularities for all $k \in \N_0$ and obviously $H^0(X,\pi_*\Omega^k_Y)\isom H^0(\tX,\Omega^k_\tX)$.
The inclusion $H^1(X,\pi_*\Omega_\tX) \subset H^1(\tX,\Omega_\tX)$ is deduced from the Leray spectral sequence.

Thus, degeneration of the spectral sequence and injectivity of $\psi$ will follow once we show that $H^1(X,\pi_*\Omega_\tX) \subset H^1(\tX,\Omega_\tX)$ has codimension equal to $m:=\dim N_1(Y/X)=\dim H^{1,1}(\tX)-\dim H^{1,1}(X)$, where $N_1(Y/X)$ is the kernel of the surjection $N_1(Y) \to N_1(X)$, see \cite[(12.1.5)]{KM} and note that log terminal singularities are rational.
But $\coker\psi^{1,1}$ is the image of $H^1(\Omega_\tX) \to H^0(R^1\pi_*\Omega_\tX)$. So let $C_1, \ldots, C_m$ be curves in $\tX$ contracted to a point under $\pi$ such that their classes form a basis of $N_1(\tX/X)$ and let $L_1, \ldots, L_m$ be line bundles on $\tX$ such that their Chern classes $\xi_i:=\chern_1(L_i)\in H^1(\Omega_\tX)$ define linearly independent functionals on $N_1(\tX/X)$. Choose irreducible components $F_1, \ldots, F_m$ of fibers of $\pi$ such that $C_i \subset F_i$ for all $i$. Take resolutions of singularities  $\nu_i:\tilde F_i \to F_i$ and curves $\tilde C_i \subset \tilde F_i$ such that ${\nu_i}_*\tilde C_i = C_i$ for all $i$. If we denote $F:=\coprod{\tilde F_i}$ and by $\nu : F\to \tX$ the composition of the resolutions with the inclusion, then by the projection formula $\nu^*\xi_i.\tilde C_j= \xi_i.C_j$ so that the $\nu^*\xi_i$ are still linearly independent. This implies that the $\xi_i$ are mapped to an $m$-dimensional subspace of $H^1(F,\Omega_F)$ under the composition $H^1(\Omega_\tX) \to H^0(R^1\pi_*\Omega_\tX)\to H^1(F,\Omega_F)$. In particular, $\rk \left(H^1(\Omega_\tX) \to H^0(R^1\pi_*\Omega_\tX)\right) \geq m$, which completes the proof of the lemma. 
\end{proof}

From the proof of the preceding lemma we deduce
\begin{corollary}\label{corollary hodge}
Let $X$ be a normal compact variety of Fujiki class with rational singularities, let $\pi:Y \to X$ be a resolution of singularities, and denote by $j:U=X^\reg\to X$ the inclusion of the regular part. Then for $k,p+q \leq 2$ we have 
\begin{enumerate}
 \item\label{corollary hodge item one} $H^{p,q}(X)\isom H^q(X,\pi_*\Omega_Y^p)\isom H^q(X,j_*\Omega_U^p)$, 
 \item $H^k(X,\C)\isom \HH^k(X,\pi_*\Omega_Y^\bullet)\isom \HH^k(X,j_*\Omega_U^\bullet)$, and 
 \item $F^pH^k(X,\C)\isom \HH^k(X,\pi_*\Omega_Y^{\geq p})\isom \HH^k(X,j_*\Omega_U^{\geq p})$.
 \end{enumerate}
\end{corollary}
\begin{proof}
The isomorphism $H^{p,q}(X)\isom H^q(X,\pi_*\Omega_Y^p)$ was shown in the proof of Lemma \ref{lemma hodge}. There, we saw that $H^{p,q}(X)\subset H^q(X,\pi_*\Omega_Y^p)$ and both spaces were shown to have the same codimension in $H^p(Y,\Omega_Y^q)$. Therefore, they are equal. The statement about the pushforward of $\Omega_U^p$ is again deduced from \cite[Corollary 1.8]{KS18}. The second statement is contained in Lemma \ref{lemma hodge} and the third statement is a consequence of it together with the first statement of the corollary.
\end{proof}
Again, the second statement was shown by Schwald in \cite[Theorem 7]{schwald} for projective varieties. Now we turn to the relative situation. The following result follows from the absolute case by homological algebra.
\begin{lemma}\label{lemma hodge relativ}
Let $X_0$ be a normal compact variety of Fujiki class with rational singularities, let $f: X \to S$ be a flat deformation of $X_0$ over a local Artinian base scheme $S$ of finite type over $\C$, let $j_0:U_0 \into X_0$ be the inclusion of the regular locus and let $U\to S$ be the induced deformation of $U_0$. Let us denote by $j:U\into X$ the inclusion and suppose that $j_*\Omega^\bullet_{U/S}$ is flat over $S$. 
Then the following hold:
\begin{enumerate}
	\item\label{lemma hodge relativ item one} For each $p \in \N_0$ and for every closed subscheme $S' \into S$ we have $\left(j_*\Omega_{U/S}^p\right) \tensor_{\sO_S} \sO_{S'} = j'_*\Omega_{U'/S'}^p$ where $j':U' \into X'$ is the base change of $j$ to $S'$. Moreover, the sheaf $j_*\Omega_{U/S}^p$ is a coherent $\sO_X$-module.
	\item\label{lemma hodge relativ item two} For all $k\leq 2$ the canonical map $H^k(X,f^{-1}\sO_S)\to \HH^k(X,j_*\Omega_{U/S}^\bullet)$ is an isomorphism and the Hodge-de Rham spectral sequence
\begin{equation}\label{eq spec seq relativ}
 E^{p,q}_1=H^q(j_*\Omega_{U/S}^p) \Rightarrow \HH^{p+q}(X,j_*\Omega_{U/S}^\bullet)
\end{equation}
degenerates on $E_1$ in the region where $p+q \leq 2$. 
	\item\label{lemma hodge relativ item three} The $\sO_S$-modules $H^q(X,j_*\Omega_{U/S}^p)$ are free for $p+q \leq 2$ and compatible with arbitrary base change.
\end{enumerate} 
\end{lemma}
\begin{proof}
We put $R:=\Gamma(S,\sO_S)$ and denote by $\gothm \subset R$ its maximal ideal. The case $R=\C$ was treated in Lemma \ref{lemma hodge}.

For the proof of \eqref{lemma hodge relativ item one}, we first note that $\left(j_*\Omega_{U/S}^p\right) \tensor_{\sO_S} \sO_{S'} = \left(j_*\Omega_{U/S}^p\right) \tensor_{j_*\sO_U} j'_*\sO_{U'} = j'_*\Omega_{U'/S'}^p$ by normality of $X_0$. It remains to prove coherence. Let us assume $\gothm\neq 0$ and denote by $n \in \N$ the unique natural number such that $\gothm^n\neq 0$ but $\gothm^{n+1}=0$. We will argue by induction on $n$. From now on, let $S' \into S$ be the closed subscheme defined by $\gothm^n$ and let $j':U' \into X' \to S'$ be the base change of $j: U \into X$ to $S'$. By normality of $X_0$ we have that $\left(j_*\Omega_{U/S}^p\right) \tensor \sO_{S'} = j'_*\Omega_{U'/S'}^p$ and thus by flatness for each $p$ there is an exact sequence
\begin{equation}\label{eq omega induction}
0\to {j_0}_*\Omega_{U_0}^p \tensor \gothm^n \to j_*\Omega_{U/S}^p \to j'_*\Omega_{U'/S'}^p \to 0
\end{equation}
Note that $\gothm^n$ is a $R/\gothm = \C$-vector space. By the inductive hypothesis, the left and the right term in the sequence are coherent, thus the same holds for the middle term.

The strategy for \eqref{lemma hodge relativ item two} and \eqref{lemma hodge relativ item three} is basically the same as in the proof of \cite[Th\'eor\`eme 5.5]{Deligne}. The differentials on all modules $E_1^{p,q}$ with $p+q=k$ will be zero if and only if $\sum_{p+q=k}\lg_R E_1^{p,q}=\lg_R \HH^k(X,j_*\Omega_{U/S}^\bullet)$ where $\lg_R$ denotes the length as an $R$-module. Note that both sides are finite.

Flatness and coherence of $j_*\Omega^\bullet_{U/S}$ entail that there is a bounded below complex $L^\bullet$ of free $R$-modules of finite rank such that there is an isomorphism $H^q(X,j_*\Omega_{U/S}^p \tensor f^*M) \isom H^q(L\tensor_R M)$ which is functorial in the $R$-module $M$, see \cite[Ch~3, Th\'eor\`eme 4.1]{BS}. This is where in the algebraic category, Deligne uses \cite[Th\'eor\`eme (6.10.5)]{EGAIII2} instead. By \cite[(3.5.1)]{Deligne}, we obtain that $\lg_R H^q(X,j_*\Omega_{U/S}^p) \leq \lg R \cdot \lg_R H^q(X_0,{j_0}_*\Omega_{U_0}^p)$ and equality holds if and only if $H^q(X,j_*\Omega_{U/S}^p)$ is $R$-free. 
 
For $k\leq 2$ we have 
\begin{equation}\label{eq deligne freeness inequality}
\begin{aligned}
\lg_R \HH^k(X,j_*\Omega_{U/S}^\bullet) &\leq  \sum_{p+q=k}\lg_R H^q(X,j_*\Omega_{U/S}^p)\\
&\leq  \lg R \cdot \sum_{p+q=k}\dim_\C H^q(X_0,{j_0}_*\Omega_{U_0}^p) \\
&= \lg R \cdot H^k(X_0,\C) 
\end{aligned}
\end{equation}
where the first inequality is the existence of the spectral sequence, the second one was explained just before and the equality is the degeneracy of the spectral sequence for $X_0$, see Lemma \ref{lemma hodge}.

As in the proof of \eqref{lemma hodge relativ item one} we may assume $\gothm\neq 0$, and we will show by induction on the minimal $n \in \N$ satisfying $\gothm^{n+1}=0$ that $\HH^k(X,j_*\Omega_{U/S}^\bullet) = H^k(X,f^{-1}\sO_S) =  H^k(X,\ul{R}_X)$ for $k \leq 2$ where $\ul{R}_X$ denotes the constant sheaf $R$. The complexes in \eqref{eq omega induction} have natural augmentations from $\ul\C_X \tensor \gothm^n$, $\ul R_X$, and $\ul R'_X$ where $R'= R/\gothm^n = \Gamma(S',\sO_{S'})$. 
Applying cohomology we obtain the following diagram with exact rows. The upper row is exact by the universal coefficient theorem. 
\[
\xymatrix@C=5mm{
0\ar[r] &  H^k\left(X_0,\C\right) \tensor \gothm^n \ar[r] \ar[d] &  H^k\left(X_0,\ul R_X \right) \ar[r]  \ar[d] &  H^k\left(X_0,\ul{R}'_{X'}\right) \ar[r]  \ar[d] &  0\\
\ldots \ar[r] &  \HH^k\left(X_0,{j_0}_*\Omega_{U_0}^\bullet\right) \tensor \gothm^n \ar[r]^(0.57){\beta_k} &  \HH^k(X,j_*\Omega_{U/S}^\bullet) \ar[r]^(0.48){\alpha_k} &  \HH^k(X',j'_*\Omega_{U'/S'}^\bullet) \ar[r] &  \ldots\\
}
\]
where the outer vertical morphisms are isomorphisms by induction. Thus, $\alpha_k$ is surjective for all $k \leq 2$. As the bottom row is part of the long exact sequence in cohomology, surjectivity of $\alpha_{k-1}$ implies injectivity of $\beta_k$. 
Thus, the middle vertical morphism is an isomorphism by the $5$-lemma. Therefore, the inequality in \eqref{eq deligne freeness inequality} is an equality which entails all the freeness statements we wanted to prove.

The base change property follows from the local freeness by \cite[Ch~3, Corollaire 3.10]{BS} (this is the analog of \cite[(7.8.5)]{EGAIII2} in the analytic case).
\end{proof}

\section{Symplectic varieties, symplectic resolutions, and periods}\label{sec symplektisch}

In this section we collect some definitions and summarize mostly known results that will be used in subsequent sections.  Some of the proven results are probably well-known to experts but we include the proofs for the precise statements that we need. In Section \ref{sect sympl def} we introduce our main object of study, singular symplectic varieties and their irreducible symplectic resolutions.  We also provide some basic results about these varieties that will be essential in Section \ref{sec defo}.  In Section \ref{section hyp hodge} we discuss properties of their periods, and the associated period domains.  In Sections \ref{sect mono orbit} and \ref{sect cone stuff} we summarize work of Verbitsky on the classification of monodromy orbit closures in the period domain and the space of K\"ahler cones.

\subsection{Symplectic varieties and symplectic resolutions}\label{sect sympl def}

Recall from \cite[Proposition 4]{B} that an \emph{irreducible symplectic manifold} is a simply connected compact K\"ahler manifold $\tX$ such that $H^0(\tX,\Omega_\tX^2)=\C \sigma$ for a holomorphic symplectic $2$-form $\sigma$. By \cite[Th\'eor\`eme 5]{B}, there is a nondegenerate quadratic form $q_\tX:H^2(\tX,\Z)\to \Z$ of signature $(3,b_2(Y)-3)$, the \emph{Beauville--Bogomolov--Fujiki form}. Thus, the associated bilinear form gives an injection $\tilde q_Y:H^2(\tX,\Z)\into H_2(\tX,\Z)$ which becomes an isomorphism over $\Q$.

We will need the notion of a K\"ahler complex space, due to Grauert, see  \cite[\S 3, 3., p. 346]{Gra}. Let $X$ be a reduced complex space. A \emph{K\"ahler form} on $X$ is  given by an open covering $X=\bigcup_{i\in I} U_i$ and smooth strictly plurisubharmonic functions $\vphi_i:U_i\to \R$ such that on $U_{ij}:=U_i\cap U_j$ the function $\vphi_i\vert_{U_{ij}} - \vphi_j\vert_{U_{ij}}$ is pluriharmonic, i.e., locally the real part of a holomorphic function. Here, a \emph{smooth function} on $X$ is by definition just a function $f:X\to \R$ such that under a local holomorphic embedding of $X$ into an open set $U\subset \C^n$, there is a smooth (i.e., $C^\infty$) function on $U$ (in the usual sense) that restricts to $f$ on $X$. A \emph{K\"ahler space} is then a reduced complex space admitting a K\"ahler form. Note that the form is not part of the structure.

We do not need many further details on K\"ahler spaces.  Indeed, we will only use that compact K\"ahler spaces can be resolved by compact K\"ahler manifolds (see e.g. \cite[II, 1.3.1 Proposition]{Var89}) and therefore come equipped with mixed Hodge structures on their cohomology groups.

\begin{definition}\label{definition isr}
We will use the term \emph{symplectic variety} in the same sense as Beauville\footnote{To be precise, Beauville studied algebraic varieties whereas we work in the complex analytic category. Apart from that the definition is however literally the same.} to mean a normal complex K\"ahler variety $X$ together with a holomorphic symplectic $2$-form $\sigma$ on its regular part that extends holomorphically to one (and hence to any) resolution of singularities, see \cite[Definition 1.1]{Be}. An \emph{(irreducible) symplectic resolution} is a resolution $\pi:Y\to X$ of a compact K\"ahler symplectic variety $X$ by an (irreducible) symplectic manifold $Y$.
\end{definition}

\begin{remark}\hspace{.1in}
\begin{enumerate}
	\item Different classes of singular symplectic varieties have been studied recently. The class of irreducible symplectic varieties defined by Greb--Kebekus--Peternell \cite[Definition 8.16]{GKP} for example is the relevant class of symplectic varieties showing up in the singular version of the Beauville-Bogomolov decomposition. The decomposition theorem has been established by H\"oring--Peternell \cite{HP19} building on work of Druel \cite{Dru} and Greb--Guenancia--Kebekus \cite{GGK}.
\item Of course a symplectic resolution of a symplectic variety is not guaranteed to exist.  Let $\pi:Y\to X$ be a resolution of singularities of a symplectic variety $X$. Then $\pi$ is a symplectic resolution if and only if it is a crepant resolution. It is well-known that not every variety has a crepant resolution, see e.g. Example \hyperref[example symplectic item noncrepant]{\ref{example symplectic}~\eqref{example symplectic item noncrepant}} below. To ask for the existence of a symplectic resolution is thus a strong assumption. We expect however, that the techniques developed here are applicable in greater generality, namely by replacing the crepant resolution by a $\Q$-factorial terminalization (for some recent progress see \cite{BL18}, especially Theorems 1.1 and 1.5). Note that every algebraic variety with klt singularities has a $\Q$-factorial terminalization by \cite[Corollary 1.4.3]{BCHM}.
\end{enumerate}

\end{remark}

\begin{Ex}\label{example symplectic}\hspace{1in}
\begin{enumerate}
\item Recall that there are the following known deformation types of irreducible symplectic manifolds, all arising from moduli spaces of sheaves on \Kthree or abelian surfaces:  for any $n>1$, the Hilbert scheme $S^{[n]}$ of length $n$ subschemes of a \Kthree surface $S$ as described in \cite[Section 6.]{B} is an irreducible symplectic manifold, which by Lemme 2 of op. cit. satisfies $b_2=23$. For any $n>1$, the generalized Kummer variety $K_n(A)$, i.e., the Hilbert scheme of zero-sum length $n+1$ subschemes of an abelian surface $A$ as described in \cite[Section 7.]{B} is irreducible symplectic and satisfies $b_2=7$ by Proposition 8 of op. cit. There are two more sporadic examples: in dimension $10$, the symplectic resolution constructed in \cite{OG1} of a singular moduli space of sheaves on a \Kthree surface $S$ with a certain Mukai vector is an irreducible symplectic manifold. It has $b_2=24$ by \cite[Theorem 1.1]{Rap08}. In dimension $6$, the symplectic resolution of a singular moduli space of trivial-determinant sheaves on an abelian surface $A$ is an irreducible symplectic manifold with $b_2=8$, see \cite[(1.4)~Theorem]{OG2}. Note that the Betti numbers are different in all these examples; as of now these are all the known examples of irreducible symplectic manifolds up to deformation.
\item Evidently, if $\pi:\tX \to X$ is a proper birational morphism from an (irreducible) symplectic manifold $\tX$ to a normal K\"ahler variety $X$, then $X$ is a symplectic variety and $\pi$ is an (irreducible) symplectic resolution.  Thus, symplectic resolutions can instead be viewed as birational contractions of symplectic manifolds.
\vskip0em\indent As all known deformation types of irreducible symplectic manifolds arise from moduli spaces of sheaves, a rich source of symplectic resolutions is given by wall-crossing contractions associated to non-generic polarizations.  We systematically investigate such examples in Section \ref{sec k3type}. 
\item Divisorial contractions yield $\Q$-factorial $X$, and these were studied in \cite{LP}, see especially Section 2 there.  The more general context dealt with in the current paper allows for small contractions, the easiest example of which is the contraction of a Lagrangian projective space to a point.  See \cite[Examples 8-10]{BB} for explicit examples.
\item Let $Y \subset  \P^5$ be a singular cubic fourfold with ADE singularities not containing a plane. Then it has been shown in \cite[Theorem 3.3]{Leh18} that the variety $M_1(Y)$ of lines on $Y$ is a symplectic variety which is birational to the second punctual Hilbert scheme of an associated K3 surface. It follows that $M_1(Y)$ admits a crepant resolution by an irreducible symplectic manifold, see \cite[Corollary 5.6]{Leh18}. A similar statement is deduced for the target space $Z(Y)$ of the MRC-fibration of the Hilbert scheme compactification of the space of twisted cubics on $Y$, see Theorem 1.1, Corollary 5.5,  and Corollary 6.2 of op. cit.
\item\label{example symplectic item noncrepant} As mentioned above, not every symplectic variety admits a symplectic resolution. For example, every $\Q$-factorial, terminal symplectic variety is either smooth or does not have any symplectic resolution. Indeed, $\Q$-factoriality implies that the exceptional locus of a resolution is a divisor and terminality then implies that every irreducible component of this divisor has strictly positive discrepancy, in particular, the resolution cannot be crepant. The $n$-th symmetric product $X^{(n)}$ with $n\geq 2$ of a smooth symplectic variety $X$ of dimension $\geq 4$ provides a concrete example, since group quotients are always $\Q$-factorial and terminality for symplectic varieties is equivalent to the singular locus having codimension $\geq 4$, see \cite[Corollary 1]{Nam01c}.
\end{enumerate}
\end{Ex}

Given a symplectic resolution $\pi:\tX \to X$, from \cite[(12.1.3) Theorem]{KM} it follows that we have a short exact sequence 
\begin{equation}\label{eq relative homology sequence}
0\to H_2(Y/X,\Q) \to H_2(Y,\Q) \to H_2(X,\Q)\to 0. 
\end{equation}

Here, $H_2(Y/X,\Q)$ denotes the subspace of $H_2(Y,\Q)$ generated by all pushforwards of cohomology classes in the fibers of $\pi$. Let us recall that $N_1(Y)$ is defined as the abelian group of algebraic $1$-cycles modulo numerical equivalence and that $N_1(Y/X) \subset N_1(Y)$ is defined as the kernel of the push forward map $\pi_*:N_1(Y) \to N_1(X)$. Then by loc. cit. $H_2(Y/X,\Q)$ is generated by algebraic cycles and therefore coincides with $N_1(Y/X)_\Q:=N_1(Y/X)\otimes\Q$. As in the introduction, we denote $\tilde q_Y:H^2(Y,\Q) \to H_2(Y,\Q)$ the isomorphism induced by the quadratic form $q_Y$. 

The following lemma is an elementary consequence of the results of Section \ref{sec hodge} and is fundamental in what follows, especially for the study of locally trivial deformations, see Section \ref{sec defo}. Among other things it will be used in the description of deformations of $Y$ that induce locally trivial deformations of $X$, see Proposition \ref{prop defo}.

\begin{lemma}\label{lemma symplektisch}
Let $\pi:\tX \to X$ be an irreducible symplectic resolution. Then $\pi^*:H^2(X,\Z) \to H^2(\tX,\Z)$ is injective and this injection is an equality on the transcendental part $H^2(X)_\tr=H^2(\tX)_\tr$. The restriction of $q_Y$ to $H^2(X,\Z)$ is nondegenerate. 
The $q_Y$-orthogonal complement to $H^2(X,\Q)$ in $H^2(\tX,\Q)$ is $N:=\left(\tilde q_Y\right)^{-1}(N_1(\tX/X)_\Q)$. Moreover, $N$ is negative definite with respect to $q_Y$.
\end{lemma}
\begin{proof}
Injectivity and equality of the transcendental parts follow from Lemma \ref{lemma peternell}. To see that $N$ is $q_Y$-orthogonal to $H^2(X,\Q)$ we only have to unravel the definition of $N$: let $C$ be a curve in $\tX$ that is contracted to a point under $\pi$. Then $D_C:=\tilde q_Y^{-1}([C])\in H^2(X,\Q)$ is the unique class such that $C.\pi^*\alpha=q_Y(D_C,\pi^*\alpha)$ for all $\alpha\in H^2(X,\Q)$. But $C.\pi^*\alpha=0$ by the projection formula, so indeed $N \perp H^2(X,\Q)$.

From \eqref{eq relative homology sequence} and the equality $H_2(Y/X,\Q)=N_1(Y/X)_\Q$, we infer that $\dim N + b_2(X) = b_2(Y)$. To conclude the proof it therefore suffices to show that $N$ is negative definite. For this we may extend the coefficients to $\R$. The restriction of $q_Y$ to $H^{1,1}(\tX)$ is nondegenerate with one positive direction and $N \subset H^{1,1}(\tX)$. We fix a K\"ahler class $h$ on $X$. Then $q_Y(h)>0$ as $h$ is big and nef on $Y$, so $h^\perp \subset H^{1,1}(\tX)$ is negative definite.  But $N \subset h^\perp$ which is what we needed to show.
\end{proof}

The following is an essential ingredient in the proof of Theorem \ref{theorem deflt is smooth} about the unobstructedness of locally trivial deformations. 
\begin{proposition}\label{proposition tangent sheaf}
Let $\pi:Y\to X$ be an irreducible symplectic resolution and let $j:U\to X$ be the inclusion of the regular locus. Then, we have 
\[
T_X=j_*T_U= j_*\Omega_U=\pi_* \Omega_{Y}.
\]
\end{proposition}
\begin{proof}
Indeed, as $X$ is integral, $T_X$ is reflexive so that the first equality follows from normality of $X$. The second equality comes from the symplectic form on the regular part, and the third is again \cite[Corollary 1.8]{KS18}.
\end{proof}

\subsection{Hodge structures of hyperk\"ahler type}\label{section hyp hodge}
Let $\Lambda$ be a free finitely-generated $\Z$-module.  Let $q$ be a nondegenerate symmetric form on $\Lambda$ of signature $(3,\rk(\Lambda)-3)$.  

\begin{definition}\label{definition hodge structure of hk type}
A \emph{Hodge structure of hyperk\"ahler type} on $(\Lambda,q)$ is a weight two integral Hodge structure $H$ on $\Lambda $ with $h^{2,0}=h^{0,2}=1$ such that $q:H\otimes H\to \Z(-2)$ is a morphism of Hodge structures and for which the restriction of $q$ to $(H^{2,0}\oplus H^{0,2})_\R$ is positive definite.
\end{definition}

Hodge structures of hyperk\"ahler type $H$ are parametrized by the period domain 
	\begin{equation}\label{eq hk period domain}
	\Omega_{\Lambda}:=\{[\sigma] \in \P(\Lambda_\C)\mid q(\sigma,\sigma)=0, q(\sigma,\bar\sigma)>0\}.
	\end{equation}
We denote by $H_p$ the Hodge structure on $\Lambda$ corresponding to $p\in \Omega_\Lambda$.  The positive real plane $(H^{2,0}\oplus H^{0,2})_\R$ is canonically oriented by $\Re(\sigma)\wedge\Im(\sigma)$ where $\sigma\in H^{2,0}$ is a generator. Conversely, a positive oriented plane $P$ in $\Lambda_\R$ determines an element $\sigma \in \Omega_\Lambda$ via $\sigma=v+iw$ where $(v,w)$ is an oriented orthonormal basis of $P$. Thus, the period domain $\Omega_\Lambda$ can alternatively be thought of as the space of oriented positive-definite planes in $\Lambda_\R$.

\subsection{Monodromy orbit closures}\label{sect mono orbit}
We will also need Verbitsky's classification \cite[Theorem 2.5]{Ve17} of orbit closures of hyperk\"ahler periods under arithmetic lattices in the orthogonal group.  We make the following definition:  
 \begin{definition}
Let $H$ be a pure weight two integral Hodge structure with underlying $\Z$-module $\Lambda$.  The \emph{rational rank} of $H$ is $\rrk(H):=\rk((H^{2,0}\oplus H^{0,2})\cap \Lambda)$.  Note that $0\leq \rrk(H)\leq 2 \cdot h^{2,0}$.  For $p\in \Omega_\Lambda$ we will also denote by $\rrk(p):=\rrk(H_p)$ the rational rank of the Hodge structure $H_p$ on $\Lambda$ corresponding to $p$.
 \end{definition}

Fixing an oriented positive-definite plane $P_0\in\Omega_\Lambda$ as a basepoint, we obtain an isomorphism
\begin{equation} \label{eq period homogeneous}
\Omega_\Lambda\cong \SO(\Lambda_\R)^\circ/\SO(P_0)\times\SO(P_0^\perp)^\circ
\end{equation}
where the superscript stands for the identity component. Let us assume from now on that $\rk(\Lambda)> 4$. This hypothesis ensures that $\SO(P_0^\perp)^\circ$ is generated by unipotents.  Now the only closed connected Lie subgroups of $\SO(3,n)^\circ$ containing $\SO(1,n)^\circ$ are $\SO(1,n)^\circ$, $\SO(2,n)^\circ$, and $\SO(3,n)^\circ$ (see \cite[Theorem 2.1]{Ve17}).  Thus, the smallest closed connected Lie subgroup of $\SO(\Lambda_\R)^\circ$ which contains $\SO(P_0^\perp)$ and is defined over $\Q$ is $\SO((\ell_0^\perp)_\R)^\circ$, where $\ell_0:=P_0\cap \Lambda$.  Note we have $\rrk(P_0)=\rk(\ell_0)$.  Given an arithmetic lattice $\Gamma \subset \O(\Lambda)$, we set $\Gamma^\circ:=\SO(\Lambda_\R)^\circ\cap \Gamma$.  By an application of Ratner's theorem (see \cite[Theorem 4.2]{V15} for the precise statement, or \cite[\S1.1.15(2)]{Mor} and \cite[\S1.1.19]{Mor} for further background), the orbit closure of $\Gamma^\circ \cdot\id \in\Gamma^\circ\backslash\SO(\Lambda_\R)^\circ$ under $\SO(P_0^\perp)^\circ$ is the (closed) orbit under $\SO((\ell_0^\perp)_\R)^\circ$.  It follows then that the orbit closure under $\SO(P_0)\times\SO(P_0^\perp)^\circ$ is either the orbit itself, $(\Gamma^\circ\cdot \id)\SO((\ell_0^\perp)^\circ_\R) \SO(P_0)$, or all of $\Gamma^\circ\backslash\SO(\Lambda_\R)^\circ$ when $\rk(\ell_0)=2,1,$ or $0$, respectively.  Note that for the middle case, $(\Gamma^\circ\cdot \id)\SO((\ell_0^\perp)^\circ_\R) \SO(P_0)$ is closed (as $(\Gamma^\circ\cdot \id)\SO((\ell_0^\perp)^\circ_\R)$ is closed and $\SO(P_0)$ compact) and contained in the orbit closure, hence equal to it.

\begin{proposition}[Theorem 2.5 in \cite{Ve17}]\label{proposition verbitsky}There are three possibilities for the orbit closure of $P_0\in\Omega_\Lambda$ under an arithmetic lattice $\Gamma\subset\O(\Lambda)$ depending on $r=\rrk(P_0)$:
 \begin{enumerate}
 \item[$(r=2)$] the orbit is closed;
 \item[$(r=1)$] the orbit closure is the union of $G_{\gamma\ell_0}$ for $\gamma\in\Gamma$, where $G_\ell$ is the subset of positive planes $P$ containing $\ell$;
 \item[$(r=0)$] the orbit is dense.
 \end{enumerate}  
 \end{proposition}
 \begin{remark}
   The $r=1$ case was omitted in \cite[Theorem 4.8]{V15} and corrected recently in \cite{Ve17}.  Note that in the $r=1$ case, $G_\ell\subset\Omega_\Lambda$ is a real-analytic submanifold of real codimension $\rk\, \Lambda-2$ in $\Omega_\Lambda$.
 \end{remark}
 
 \subsection{Decompositions of the positive cone}\label{sect cone stuff}
By $\Lambda$ we still denote a lattice of signature $(3,\rk(\Lambda)-3)$. Let $\Gamma\subset\O(\Lambda)$ be an arithmetic lattice.
\begin{definition}\label{def mbm coll}  An \emph{MBM collection} of $\Lambda$ (for $\Gamma$) is a $\Gamma$-invariant set $M\subset \Lambda-\{0\}$ of classes of bounded negative square.
\end{definition}

\begin{setup}\label{setup cones}
In the following, we refer to \cite[\S4]{markmantor} for details.  We denote by $C(\Lambda)\subset \Lambda_\R$ the cone of positive vectors. It has the homotopy type of a $2$-sphere, in particular, it is connected. For a hyperk\"ahler Hodge structure $H$ on $\Lambda$, we define $C^{1,1}(H)\subset H_\R^{1,1}$ to be the cone of positive vectors in the $(1,1)$-part. This cone has two components, and a choice of component is equivalent to a choice of generator of $H^2(C(\Lambda),\Z)$. We fix such a generator of $H^2(C(\Lambda),\Z)$ and denote by  $C^{1,1}(H)^+$ the corresponding component of $C^{1,1}(H)$. Let $M$ be an $MBM$ collection of $\Lambda$ for $\Gamma$. For each $p\in\Omega_\Lambda$, the collection $M$ induces a wall-and-chamber decomposition of $C^{1,1}(H_p)^+$ whose open chambers are the connected components of 
\begin{equation}\label{eq kaehler type chambers}
C^{1,1}(H_p)^+-\bigcup_{\alpha\in M_p} \alpha^\perp
\end{equation}
where $M_p=M\cap H^{1,1}_p$.
\end{setup}

\begin{remark}\label{remark MBM gen}Note that the collection of hyperplanes $\bigcup_{\alpha\in M_p} \alpha^\perp$ is locally finite in $C^{1,1}(H_p)^+$ since the squares of the classes in $M$ are bounded.  Indeed, for any $\omega\in C^{1,1}(H_p)^+$ the space $\omega^\perp\cap (H^{1,1}_p)_\R$ is negative definite, and the fixed radius closed balls in $\omega^\perp\cap (H^{1,1}_p)_\R$ form a proper family over $C^{1,1}(H_p)^+$, hence only finitely many walls $\alpha^\perp$ can intersect any given compact subset of $C^{1,1}(H_p)^+$.\end{remark}

\begin{definition}\label{definition kaehler type chambers}
The open chambers of $C^{1,1}(H_p)^+$ defined by \eqref{eq kaehler type chambers} are called \emph{K\"ahler-type chambers}. The hyperplanes $\alpha^\perp$ for $\alpha \in M_p$ are called \emph{K\"ahler-type walls}.
\end{definition}

Let $\scrC^{1,1}(\Omega_\Lambda)^+$ be the universal type (1,1) positive cone---that is, the subspace of the product $\Omega_\Lambda\times C(\Lambda)$ consisting of pairs $(p,\omega)$ for which $\omega\in C^{1,1}(H_p)^+$.  Define
 \[
\Omega^\mathrm{class}_\Lambda:=\scrC^{1,1}(\Omega_\Lambda)^+-\bigcup_{\alpha\in M}\alpha^\perp\times\alpha^\perp.
\]  
As in Remark \ref{remark MBM gen}, for any $(p,\omega)\in \scrC^{1,1}(\Omega_\Lambda)^+$ the fixed radius closed balls in $\omega^\perp\cap (H^{1,1}_p)_\R$ form a proper family over $\scrC^{1,1}(\Omega_\Lambda)^+$, so the collection $\bigcup_{\alpha\in M}\alpha^\perp\times\alpha^\perp$ of universal hyperplanes is locally finite in $\scrC^{1,1}(\Omega_\Lambda)^+$ and in particular $\Omega_\Lambda^\mathrm{class}$ is open in $\scrC^{1,1}(\Omega_\Lambda)^+$.

 Evidently, the fiber of $\Omega_\Lambda^\mathrm{class}$ over $p\in\Omega_\Lambda$ is the union of the open chambers of the wall-and-chamber decomposition of $C^{1,1}(H_p)^+$ defined above. We define the space $\Omega^\mathrm{cone}_\Lambda$ to consist of pairs $(p,C)$ for $p\in\Omega_\Lambda$ and $C$ an open chamber of the decomposition of $C^{1,1}(H_p)^+$, topologized as a quotient of $\Omega^\mathrm{class}_\Lambda$.  The forgetful map $\Omega_\Lambda^\mathrm{cone}\to\Omega_\Lambda$ is continuous since $\Omega^\mathrm{class}_\Lambda\to\Omega_\Lambda$ is.  Moreover, for any $p\in \Omega_\Lambda$ and any $\omega$ in an open chamber of $C^{1,1}(H_p)^+$, a sufficiently small ball $B$ around $(p,\omega)$ in $\scrC^{1,1}(\Omega_\Lambda)^+$ is contained in the open subset $\Omega^\mathrm{class}_\Lambda$.  Such a $B$ therefore defines a section of $\Omega_\Lambda^\mathrm{cone}\to\Omega_\Lambda$ over its image in $\Omega_\Lambda$, and we have proven the following statement.

\begin{proposition}\label{proposition local torelli cone space version}
The forgetful map $\Omega_\Lambda^\mathrm{cone}\to\Omega_\Lambda$ is a local isomorphism.\qed
\end{proposition}

Note that both $\Omega^\mathrm{class}_\Lambda$ and $\Omega^\mathrm{cone}_\Lambda$ depend on the choice of $MBM$ collection $M$, and that $\Gamma$ acts on both by continuous maps.  We have the following formulation of result of Verbitsky:
\begin{theorem}[Theorem 3.1 of \cite{Ve17}]\label{theorem lift density} Assume $\rk(\Lambda)>4$ and let $M$ be an $MBM$ collection of $\Lambda$ for $\Gamma$.  Then the forgetful map $\pi:\Omega_{\Lambda}^\mathrm{cone}\to \Omega_\Lambda$ commutes with closures.
\end{theorem}
\begin{corollary}\label{corollary lift density} Assume $\rk(\Lambda)>4$. For any $x=(p,C)\in\Omega_\Lambda^\mathrm{cone}$ of non-maximal Picard rank, we have $\overline{\Gamma x}=\pi^{-1}\left(\overline{\Gamma p}\right)$.
\end{corollary}
\begin{proof}
In view of Theorem \ref{theorem lift density}, it suffices to show the inclusion $\supseteq$. Let $y\in \overline{\Gamma x}$, and suppose $y'\in \pi^{-1}(\pi(y))$, in other words, $y'$ is inseparable from $y$.  Let $U'$ and $U$ be open neighborhoods of $y'$ and $y$, respectively, which we may assume have the same image in $\Omega_\Lambda$. Then $U'$ and $U$ must meet at every point in $U$ with Picard rank zero, and therefore meet at a point in $\overline{\Gamma x}$. Hence every open neighborhood of $y'$ meets $\ol{\Gamma x}$ which proves the claim.
\end{proof}

We will finally need to understand how the cone decompositions defined above restrict to sublattices.  
\begin{definition}\label{def induced mbm}Suppose $M'$ is an $MBM$ collection of $\Lambda' $ for $\Gamma'$, and that $\Lambda\subset\Lambda'$ is a primitive sublattice for which $N:=\Lambda^\perp$ is negative definite and rationally generated by classes in $M'$.  Let $\Gamma\subset\Gamma'$ be the stabilizer of $N$.  We define the \emph{induced $MBM$ collection} of $\Lambda$ (for $\Gamma$) to be the set $M\subset\Lambda-\{0\}$ of primitive vectors of negative square which are up to scaling the orthogonal projections of classes in $M'-N$. 
\end{definition}

Evidently $M$ is invariant under the subgroup $\Gamma$. Note that the order of the scaling is bounded by the size of the discriminant group of $\Lambda$, and it follows that the classes of $M$ are indeed of bounded square. Thus, $M$ is again an MBM collection in the sense of Definition \ref{def mbm coll}. Note that in Definition \ref{def induced mbm} it is necessary to impose negativity of the square as the following example shows. 
\begin{example}
Let $\Lambda'$ be a lattice containing a sublattice spanned by two $(-2)$-classes $c_1, c_2$ with $c_1.c_2=2n$. We let $M'$ denote the $\Gamma$-orbit of $\{c_1,c_2\}$. If $N=\Z c_2$ and $\Lambda=N^\perp$, then $c_1+nc_2$ has square $2(n^2-1)$ and is the orthogonal projection of $c_1$. Note that this situation is realized geometrically by K3 surface with two smooth rational curves intersecting in $2n$ points.
\end{example}

\begin{setup}\label{setup mbm sublattice}
 Let us fix an MBM collection $M'\subset \Lambda'$, a primitive sublattice $\Lambda \subset \Lambda'$ such that $N:=\Lambda^\perp$ is negative definite and rationally generated by classes in $M'$, and let us denote by $M \subset \Lambda$ the induced MBM collection. We define $\Omega_{\Lambda'}^\mathrm{class}, \Omega_{\Lambda'}^\mathrm{cone}$ (respectively $\Omega_{\Lambda}^\mathrm{class}, \Omega_{\Lambda}^\mathrm{cone}$) as before using the classes $M'$ (respectively $M$).  
\end{setup}

\begin{remark}
Given a point $p\in \Omega_\Lambda \subset \Omega_{\Lambda'}$, we denote the corresponding Hodge structures on $\Lambda$ and $\Lambda'$ by $H_p$ and $H'_p$. Given Setup \ref{setup mbm sublattice}, we explicitly describe the relation between the K\"ahler-type chambers in $C^{1,1}(H_p)^+$ and in $C^{1,1}(H'_p)^+$, see Definition \ref{definition kaehler type chambers}.  The classes in $M_p$ are the projections of classes in $M'_p$ since $N\subset (H'_p)^{1,1}_\R$.  Thus, the walls of $C^{1,1}(H_p)^+$ are the intersections of the walls of $C^{1,1}(H'_p)^+$ with $\Lambda$.  Moreover, as $N$ is spanned by classes in $M'$, no open chamber $C$ of $C^{1,1}(H'_p)^+$ intersects $\Lambda_\R$.  It follows that each \emph{closed} chamber of $C^{1,1}(H_p)^+$ is of the form $\bar C\cap \Lambda_\R$ for an open chamber $C$ of $C^{1,1}(H_p')^+$ for which $\bar C\cap \Lambda_\R$ has nonempty interior (i.e. for which $\Lambda_\R$ is a face of $C$).
\end{remark}

\begin{definition}\label{definition omega lambda lamba prime}
Let $\Omega_{\Lambda',\Lambda}^\mathrm{class}\subset \Omega_{\Lambda'}^\mathrm{class}$ be the subspace of points $(p,\omega)$ for which $p\in\Omega_\Lambda$ and the chamber of $C^{1,1}(H_p)^+$ containing $\omega$ has $\Lambda_\R$ as a face.   Let $\Omega_{\Lambda',\Lambda}^\mathrm{cone}\subset \Omega_{\Lambda'}^\mathrm{cone}$ be the subspace of points $(p,C)$ for which $p\in\Omega_\Lambda$ and $C$ has $\Lambda_\R$ as a face. 
\end{definition}

\begin{definition}\label{definition resolution chambers}
Let us fix $p\in\Omega_\Lambda\subset\Omega_{\Lambda'}$. In Setup \ref{setup mbm sublattice}, the \emph{resolution chambers} of $C^{1,1}(H'_p)^+$ (with respect to $N$) are the finitely many chambers cut out by the finitely many hyperplanes $\alpha'^\perp$ for $\alpha'\in M'\cap N$, i.e.  the connected components of
\begin{equation}\label{eq resolution chambers}
C^{1,1}(H_p')^+-\bigcup_{\alpha\in M'\cap N} \alpha^\perp.
\end{equation}
Note that any such resolution chamber $\tau$ uniquely determines a resolution chamber in $C^{1,1}(H_t')^+$ for all $t\in \Omega_\Lambda$ which we will also denote by $\tau$. This is because these chambers are cut out by the same inequalities $q(\omega,\alpha')>0$ or $q(\omega,\alpha')<0$ for $\alpha' \in M'\cap N$ independently of $t \in \Omega_\Lambda$. We define $\Omega^\mathrm{class}_{\Lambda',\Lambda}(\tau)\subset \Omega_{\Lambda',\Lambda}^\mathrm{class}$ to be the subspace of $(p,\omega)$ with $\omega \in \tau$. We let $\Omega^\mathrm{cone}_{\Lambda',\Lambda}(\tau)\subset \Omega_{\Lambda',\Lambda}^\mathrm{cone}$ be the image of $\Omega_{\Lambda',\Lambda}^\mathrm{class}(\tau)$.
\end{definition}

Note that all resolution chambers have $\Lambda_\R$ as a face. For a very general $p\in\Omega_\Lambda\subset\Omega_{\Lambda'}$ it follows from comparison of \eqref{eq kaehler type chambers} and \eqref{eq resolution chambers} that the resolution chambers are exactly the K\"ahler-type chambers.

\begin{proposition}\label{prop face map}The connected components of $\Omega_{\Lambda',\Lambda}^\mathrm{cone}$ are the spaces $\Omega^\mathrm{cone}_{\Lambda',\Lambda}(\tau)$ for all possible resolution chambers $\tau$.  Moreover, the ``face" map
\[\Omega_{\Lambda',\Lambda}^\mathrm{cone}\to\Omega_\Lambda^\mathrm{cone}:(p,C)\mapsto (p,\mathrm{int}(\bar C\cap \Lambda_\R))\]
is an isomorphism of each component onto $\Omega_\Lambda^\mathrm{cone}$. 
\end{proposition}
\begin{proof}  As each chamber of $C^{1,1}(H_p')^+$ having $\Lambda_\R$ as a face is contained in a resolution chamber, it follows that $\Omega_{\Lambda',\Lambda}^\mathrm{cone}$ is covered by the $\Omega^\mathrm{cone}_{\Lambda',\Lambda}(\tau)$.  Replacing the strict inequalities defining $\Omega^\mathrm{class}_{\Lambda',\Lambda}(\tau)$ with nonstrict ones defines the same subset, so $\Omega^\mathrm{class}_{\Lambda',\Lambda}(\tau)$ is open and closed in $\Omega^\mathrm{class}_{\Lambda',\Lambda}$, as both are topologized as subsets of $\Omega_{\Lambda'}\times C(\Lambda')$.  As $\Omega^\mathrm{class}_{\Lambda',\Lambda}(\tau)$ is saturated with respect to the quotient $\Omega^\mathrm{class}_{\Lambda',\Lambda}\to\Omega^\mathrm{cone}_{\Lambda',\Lambda}$, it follows that the image $\Omega^\mathrm{cone}_{\Lambda',\Lambda}(\tau)$ is also open and closed in $\Omega_{\Lambda',\Lambda}^\mathrm{cone}$.  

We now claim the face map restricts to a bijection $\Omega^\mathrm{cone}_{\Lambda',\Lambda}(\tau)\to\Omega^\mathrm{cone}_\Lambda$.  Indeed, for any open chamber $C$ of $C^{1,1}(H_p)^+$ for $p\in\Omega_\Lambda$, if $\omega\in C$ then for a sufficiently small open neighborhood $B$ of $\omega$ in $C^{1,1}(H'_p)^+$, the only walls $B$ meets are those of the form $\alpha'^\perp$ for $\alpha'\in N$. As $B$ must intersect $\tau$, there is a unique chamber $C'\subset C^{1,1}(H'_p)^+$ containing $\tau\cap B$, and this must be the unique preimage of $(p,C)$ in $\Omega^\mathrm{cone}_{\Lambda',\Lambda}(\tau)$.

Noting that $\Omega_\Lambda^\mathrm{cone}$ is connected (as the very general fiber over $\Omega_\Lambda$ is a point), for the remainder of the claim it suffices to show $\Omega^\mathrm{cone}_{\Lambda',\Lambda}\to\Omega_\Lambda$ is a local isomorphism.  As in the argument for Proposition \ref{proposition local torelli cone space version}, for any $p\in\Omega_\Lambda$ and any lift $(p,\omega)\in \Omega^\mathrm{cone}_{\Lambda',\Lambda}$, a small ball $B$ around $(p,\omega)$ is contained in $\Omega^\mathrm{cone}_{\Lambda',\Lambda}$ and defines a section of $\Omega^\mathrm{cone}_{\Lambda',\Lambda}\to\Omega_\Lambda$ over its image.
\end{proof}

\section{Deformations}\label{sec defo}
As usual in deformation theory, when we speak about the semi-universal deformation $\scrZ \to \Def(Z)$ of a complex space $Z$, then the complex space $\Def(Z)$ has a distinguished point $0\in\Def(Z)$ such that the fiber of $\scrZ\to \Def(Z)$ over $0$ is $Z$ and we should actually speak about the morphism of space germs $(\scrZ,Z) \to (\Def(Z),0)$. All deformation theoretic statements have to be interpreted as statements about germs.

Let $X$ be a normal compact complex variety with rational singularities and let $\pi:Y \to X$ be a resolution of singularities.
Recall that by \cite[Proposition 11.4]{KM} there is a morphism $p:\Def(Y)\to \Def(X)$ between the Kuranishi spaces of $Y$ and $X$ and also between the semi-universal families $\scrY \to \Def(\tX)$ and $\scrX \to \Def(X)$ fitting in a diagram
\begin{equation}
\label{eq defo diag}
\xymatrix{
\scrY \ar[d]\ar[r]^{P}& \scrX \ar[d]\\
\Def(\tX) \ar[r]^{p} & \Def(X) \\
}
\end{equation}
Recall from \cite[(0.3) Corollary]{FK} that there exists a closed complex subspace $\Def^\lt(X)\subset \Def(X)$ parametrizing locally trivial deformations of $X$. More precisely, the restriction of the semi-universal family to this subspace, which by abuse of notation we denote also by $\scrX \to \Def^\lt(X)$, is a locally trivial deformation of $X$ and is semi-universal for locally trivial deformations of $X$.

Let $\pi:Y \to X$ be an irreducible symplectic resolution with $X$ of dimension $2n$. As $H^0(T_Y)=0$, every semi-universal deformation of $Y$ is universal. We also have $H^0(T_X)=H^0(\pi_*\Omega_Y)=0$ by Proposition \ref{proposition tangent sheaf} so that every semi-universal deformation of $X$ is universal. 
Let us fix universal deformations of $X$ and $Y$ and a diagram as \eqref{eq defo diag}. It is well-known that $\scrY \to \Def(Y)$ is a family of irreducible symplectic manifolds, at least in the sense of germs, i.e., possibly after shrinking the representative of $\Def(Y)$. If $X$ is projective, then also $\scrX\to\Def(X)$ is a family of symplectic varieties admitting irreducible symplectic resolutions by \cite[Theorem 2.2]{Na01}. 
We will see in Proposition \ref{prop namikawa general} that as an application of our results this statement also holds without the projectivity assumption. Our first goal is to address smoothness of the space of locally trivial deformations.
Recall from the introduction and Lemma \ref{lemma symplektisch} that 
we have an orthogonal decomposition
\begin{equation}\label{eq n}
H^2(\tX,\Q)=H^2(X,\Q)\oplus N 
\end{equation}
where $N$ corresponds under the isomorphism $\tilde q_Y:H^2(Y,\Q) \to H_2(Y,\Q)$ to the curves contracted by $\pi$ and put $m:=\dim N$.
\begin{theorem}\label{theorem deflt is smooth}
Let $X$ be a symplectic variety admitting an irreducible symplectic resolution. Then the space $\Def^\lt(X)$ of locally trivial deformations of $X$ is smooth of dimension $h^{1,1}(X) = h^{1,1}(\tX) - m$. 
\end{theorem}
\begin{proof}
Smoothness is shown using the $T^1$-lifting principle of Kawamata-Ran \cite{Ran,Ka1,Ka2}, in particular Theorem 1 of \cite{Ka1}. We refer to \cite[\S 14]{GHJ} or \cite[VI.3.6]{mydiss} for more detailed introductions. The tangent space to $\Def^\lt(X)$ at the origin is $H^1(T_X)$ which thanks to Proposition \ref{proposition tangent sheaf} can be identified with $H^1(j_*\Omega_U)$ where $j:U=X^\reg \to X$ is the inclusion. For the $T^1$-lifting property one has to show that for every infinitesimal locally trivial deformation $\sX \to S$ of $X$ over an Artinian base scheme $S$ the space $H^1(T_{\sX/S})$ is locally $\sO_S$-free and compatible with arbitrary base change. We denote again by $j:\sU \to \sX$ the inclusion of the smooth locus of $\sX \to S$.  Take an extension $\sigma \in H^0(\sU,j_* \Omega_{\sX/S}^2)$ of the symplectic form on $U\subset X$. It remains nondegenerate and hence yields an isomorphism $T_{\sX/S}\to j_*\Omega_{\sX/S}$, consequently also $H^1(T_{\sX/S})\isom H^1(\Omega_{\sX/S})$ which is free by Lemma \ref{lemma hodge relativ}. Thus, it satisfies the $T^1$-lifting property and thus the space $\Def^\lt(X)$ is smooth.
It follows from Corollary \ref{corollary hodge} that $\dim H^1(T_X) = h^{1,1}(X)$ which shows the dimension statement and completes the proof.
\end{proof}
It is convenient to introduce the following terminology:
\begin{definition}\label{definition deformable}
An irreducible symplectic resolution $\pi:Y\to X$ is \emph{deformable} if $\Def^\lt(X)$ is contained in the image of the map $p:\Def(Y)\to\Def(X)$ considered in \eqref{eq defo diag}.
\end{definition}

An easy consequence of Namikawa's work, Proposition \ref{proposition deformation remains isv} below says that any irreducible symplectic resolution of a \emph{projective} symplectic variety is deformable. We will in fact see in Proposition \ref{proposition nonproj deform} that all irreducible symplectic resolutions are deformable.
\begin{proposition}\label{proposition deformation remains isv}
Let $\pi:Y \to X$ be an irreducible symplectic resolution and consider the corresponding diagram \eqref{eq defo diag}. Then for every $t\in \Def(Y)$, the morphism $P_t:\scrY_t \to \scrX_{p(t)}$ is an irreducible symplectic resolution. In particular, if $\pi$ is deformable, any small locally trivial deformation of $X$ is again a compact K\"ahler symplectic variety admitting an irreducible symplectic resolution. Moreover, any irreducible symplectic resolution of a projective symplectic variety $X$ is deformable.
\end{proposition}
\begin{proof}
The variety $X$ has rational singularities, hence $\scrX \to \Def(X)$ is a family of K\"ahler varieties by \cite[Proposition 5]{Na01a}. As $\scrY \to \Def(X)$ is a family of irreducible symplectic manifolds, $\scrY \to\scrX$ is fiberwise an irreducible symplectic resolution over the image of $p$. The second statement follows now directly from the definition of deformability.

Suppose now that $X$ is projective. Then by \cite[Theorem (2.2)]{Na01}, the spaces $\Def(Y)$ and $\Def(X)$ are smooth of the same dimension and the map $p$ from diagram \eqref{eq defo diag} is finite. In particular, $p$ is surjective and $\Def^\lt(X)$ is contained in the image.
\end{proof}
\begin{corollary}\label{corollary projective new}
 Let $\pi:Y \to X$ be a deformable irreducible symplectic resolution and $\scrY\to \scrX$ the base-change of the top map of \eqref{eq defo diag} to $p^{-1}(\Def^\lt(X))$.  Then points $t\in p^{-1}(\Def^\lt(X))$ for which $\scrY_t$ and $\scrX_t$ are both projective are dense in every positive dimensional subvariety of $\deflt$.
\end{corollary}
\begin{proof}
By \cite[Proposition 26.6]{GHJ} the claim is true for $\scrY_t$.  Since $\scrX_t$ has rational singularities, see e.g. \cite[Theorem 3.3.3]{Kirschner}, and is K\"ahler, it is projective whenever $\scrY_t$ is by Namikawa's result \cite[Corollary 1.7]{Na02}.
\end{proof}
Next, we describe $p^{-1}(\Def^\lt(X)) \subset \Def(\tX)$.
For $N\subset H^2(Y,\Q)$ as in \eqref{eq n} let us denote by $\Def(\tX,N)\subset \Def(\tX)$ the subspace of those deformations of $\tX$ where all line bundles on $Y$ with first Chern class in $N$ deform along. This is also the subspace where $N$ remains of type $(1,1)$. It is a smooth submanifold of $\Def(\tX)$ of codimension $m=\dim N$ by \cite[1.14]{Huy99}.
\begin{proposition}\label{prop defo}
Let $\pi: \tX\to X$ be a deformable irreducible symplectic resolution. Let $\scrY \to \Def(\tX,N)$ and $\scrX \to\Def^\lt(X)$ be the (restrictions of the) universal deformations. Then for the natural morphism $p:\Def(\tX)\to\Def(X)$ we have $p^{-1}(\Def^\lt(X)) = \Def(\tX,N)$, and \eqref{eq defo diag} restricts to a diagram \begin{equation}\label{eq diag lt}
\xymatrix{
\scrY \ar[d]\ar[r]^P& \scrX \ar[d]\\
\Def(\tX,N) \ar[r]^{p} & \Def^\lt(X). \\
}
\end{equation}
Moreover, $p:\Def(\tX,N) \to \Def^\lt(X)$ is an isomorphism.
\end{proposition}
\begin{proof}
By Proposition \ref{proposition deformation remains isv} we know that for each $t\in \Def(Y)$ mapping to $\Def^\lt(X)$ the morphism $P_t:\scrY_t \to \scrX_{p(t)}$ is an irreducible symplectic resolution. By Lemma \ref{lemma hodge relativ} the second cohomology groups of locally trivial deformations of $X$ form a vector bundle on $\Def^\lt(X)$, in particular, $h^{1,1}(\scrX_{p(t)})=h^{1,1}(X)$. Thus, by the decomposition $H^2(\tX,\C)=N\oplus H^2(X,\C)$ from Lemma \ref{lemma symplektisch} we see that the space $N_1(\scrY_t/\scrX_{p(t)})$ of curves contracted by $P_t$ has dimension $m$ for all $t\in p^{-1}(\Def^\lt(X))$. As $N$ is the orthogonal complement of $H^2(X,\C)$, it also  varies in a local system. This shows the sought-for equality.

One shows as in \cite[Proposition 2.3 (ii)]{LP} that  $p$ is an isomorphism, we only sketch this: It suffices to show that the differential $T_{p,0}:T_{\Def(\tX,N),0}\to T_{\Def^\lt(X),0}=H^1(T_X)$ is an isomorphism. 
We know from \cite[(1.8) and (1.14)]{Huy99} that $T_{\Def(\tX,N),0} \subset H^1(T_\tX)$ can be identified with the orthogonal complement to $N \subset H^{1,1}(\tX)$ under the isomorphism $H^1(T_\tX) \isom H^{1,1}(\tX)$ induced by the symplectic form. In other words, $T_{\Def(\tX,N),0} \isom$ $H^{1,1}(X) \subset H^{1,1}(\tX)$. That this is mapped to $H^1(T_X)\isom H^1(j_*\Omega_U)$ under the restriction of $T_{p,0}:H^1(T_\tX)\to \Ext^1(\Omega_X,\sO_X)$ is easily verified.
\end{proof}
\begin{remark}\label{singular BBF}
Recall from \cite[Theorem 4.7]{Fujiki} that for an irreducible symplectic manifold $Y$ of dimension $2n$ there is a deformation invariant constant $c_Y$ such that $q_Y(\alpha)^n=c_Y\cdot\int_Y\alpha^{2n}$ for any $\alpha\in H^2(Y,\Z)$.  Now if $\pi:Y\to X$ is an irreducible symplectic resolution, it follows that the restriction $q_X(\beta):=q_Y(\pi^*\beta)$ only depends on $X$, since all irreducible symplectic resolutions are deformation-equivalent by \cite[Theorem 2.5]{Huy} and $\int_Y(\pi^*\beta)^{2n}=\int_X\beta^{2n}$.  Furthermore, from Lemma \ref{lemma symplektisch} we have that $q_X$ is nondegenerate of signature $(3,b_2(X)-3)$, and by Proposition \ref{prop defo} it is locally trivially deformation-invariant.  We refer to $q_X$ as the Beauville--Bogomolov--Fujiki form of $X$.
\end{remark}

Using Beauville--Bogomolov--Fujiki form and Proposition \ref{prop defo}, we obtain a singular version of the local Torelli theorem for locally trivial deformations as a direct corollary of the local Torelli theorem for a resolution \cite[Th\'eor\`eme 5]{B}.
\begin{proposition}[Local Torelli Theorem]\label{proposition local torelli}
Let $X$ be a symplectic variety admitting a deformable irreducible symplectic resolution, let $q_X$ be its Beauville--Bogomolov--Fujiki form, and let 
\begin{equation}\label{eq local period domain}
\Omega(X):=\{[\sigma]\in \P(H^2(X,\C))\mid q_X(\sigma)=0, q_X(\sigma,\bar\sigma)>0\} \subset \P(H^2(X,\C)) 
\end{equation}
be the period domain for $X$. If $f:\scrX \to \Def^\lt(X)$ denotes the universal locally trivial deformation of $X$ and $X_t:=f^{-1}(t)$, then the period map 
\[
\wp:\Def^\lt(X) \to \Omega(X), \quad t \mapsto H^{2,0}(X_t)
\]
is a local isomorphism.\qed
\end{proposition}
It should be mentioned that Namikawa has proven a local Torelli theorem for certain singular projective symplectic varieties in \cite[Theorem 8]{Na01a} and this has been generalized by Kirschner \cite[Theorem 3.4.12]{Kirschner} to a larger class of varieties.  In particular, Kirschner has proven a local Torelli theorem in the context of symplectic compact K\"ahler spaces. Let us emphasize, however, that neither Namikawa's version nor Kirschner's is what we need as they do not make any statement about local triviality. Also observe that our version of local Torelli---unlike Namikawa's or Kirschner's---does not make any assumption on the codimension of the singular locus of the variety $X$.
\begin{remark}\label{remark identify base}
In view of the local Torelli theorem \ref{proposition local torelli}, we will from now on identify the spaces $\Def(\tX,N)$ and $\Def^\lt(X)$ via the morphism $p$ if we are given a birational contraction. 
\end{remark}

The following result is a singular analog of \cite[Theorem 2.5]{Huy} which we also use in the proof.

\begin{theorem}\label{theorem huybrechts}
Let $\pi:Y \to X$ and $\pi':Y' \to X'$ be deformable irreducible symplectic resolutions. Assume that there is a birational map $\phi: Y \ratl Y'$ such that the induced map $\phi^*: H^2(\tX',\C) \to H^2(\tX,\C)$ sends $H^2(X',\C)$ isomorphically to $H^2(X,\C)$.  Then there is an isomorphism $\vphi:\Def^\lt(X)\to \Def^\lt(X')$ such that for each $t\in \Def^\lt(X)$ we have a birational map $\phi_t: \scrX_t \ratl \scrX'_{\vphi(t)}$. In particular, for general $t \in \Def^\lt(X)$ the map $\phi_{t}$ is an isomorphism, and $X$ and $X'$ are locally trivial deformations of one another.
\end{theorem}
Note that as a birational map between smooth $K$-trivial varieties, $\phi$ has to be an isomorphism in codimension one and therefore $\phi^*: H^2(\tX',\C) \to H^2(\tX,\C)$ has to be an isomorphism. For singular varieties, this last conclusion is stronger than being an isomorphism in codimension one. It is for example violated if $\pi:X\to X'$ is a small contraction of an irreducible symplectic manifold $X$.

\begin{proof}
The birational  map $ \phi:\tX \ratl \tX'$ between irreducible symplectic manifolds induces an isomorphism between $H^2(\tX,\Z)$ and $H^2(\tX',\Z)$ compatible with the Beauville--Bogomolov--Fujiki forms and the local Torelli theorem gives an isomorphism $\Def(\tX)\to \Def(\tX')$. 
Let us consider the orthogonal decompositions $H^2(\tX,\Q)=H^2(X,\Q)\oplus N$ and $H^2(\tX',\Q)=H^2(X',\Q)\oplus N'$ as in \eqref{eq n}.
By hypothesis, this decomposition is respected by $\phi^*$ from which we infer that the isomorphism $\Def(\tX)\to \Def(\tX')$ restricts to an isomorphism $\Def(\tX,N) \to \Def(\tX',N')$ and therefore yields an isomorphism $\vphi:\Def^\lt(X) \to \Def^\lt(X')$ via Proposition~\ref{prop defo}. 
It remains to show the existence of a birational  map $\phi_t: \scrX_t \ratl \scrX'_{\vphi(t)}$ for $t\in \Def^\lt(X)$ which is an isomorphism at the general point. We will identify the spaces $$S:=\Def^\lt(X)\isom\Def^\lt(X')\isom \Def(\tX,N)\isom \Def(\tX',N')$$ and consider the universal families
\[
 \scrY \to \scrX \to S \ot {\scrX'} \ot {\scrY'}
\]
from Proposition \ref{prop defo}. For any point $t\in S$, the fibers $\scrY_t$ and $\scrY'_{t}$ are deformation equivalent (by Huybrechts' theorem, see \cite[Theorem 2.5]{Huy}) and have the same periods, hence they are birational by Verbitsky's global Torelli theorem \cite[Theorem 1.17]{V13}. By countability of components of the Douady space, there is a cycle $\Gamma \subset \scrY \times_S \scrY'$ such that over a Zariski open $U \subset S$ the fiber of $\Gamma$ over $S$ is the graph of a birational map. For each $t\in S$, the cycles $\Gamma_t$ induce Hodge isometries $[\Gamma_t]_*:H^2(\scrY_t,\Q) \to H^2(\scrY'_t,\Q)$ which form a morphism of local systems when $t$ varies. Choosing $t \in S$ very general, we see that these cycles have to send $H^2(\scrX_t,\Q)$ isomorphically to $H^2(\scrX'_t,\Q)$ or equivalently, $N$ to $N'$.

The image of $\Gamma$ in $\scrX\times_S \scrX'$ is a cycle whose fiber for $t\in U$ is the graph of a birational map $\psi_t:\scrX_{t}\ratl \scrX'_{t'}$. As $[\Gamma_t]_*$ sends $H^2(\scrX_t,\Q)$ isomorphically to $H^2(\scrX_t',\Q)$, we see that ${\psi_t}_*$ sends $\Pic(\scrX_t)\tensor \Q$ isomorphically to $\Pic(\scrX'_t)\tensor \Q$. If we choose $t\in U$ such that  $\scrX_t$ and $\scrX_t'$ are projective of Picard number one, then $\psi^*$ of the ample generator of $\Pic(\scrX_t')$ is again an ample line bundle on $\scrX_t$. Therefore, $\psi_t$ must be regular and hence an isomorphism. This completes the proof.
\end{proof}

The following result is a singular analog of \cite[Theorem 4.6]{Huy99} and is interesting in its own right. The same hypotheses as in Theorem \ref{theorem huybrechts} are needed.

\begin{theorem}\label{theorem huybrechts strong}
Let $X$ and $X'$ be projective symplectic varieties with irreducible symplectic resolutions $\pi:Y\to X$ and $\pi':Y'\to X'$.  Let $\phi: Y \ratl Y'$ be a birational map such that the induced map $\phi^*: H^2(\tX',\C) \to H^2(\tX,\C)$ sends $H^2(X',\C)$isomorphically to $H^2(X,\C)$. Then there are one parameter locally trivial deformations $f:\scrX \to \Delta$, $f':\scrX' \to \Delta$ of $X$ and $X'$ such that $\scrX$ and $\scrX'$ are birational over $\Delta$ and such that $\scrX^* = f^{-1}(\Delta^\times) \isom (f')^{-1}(\Delta^\times) = (\scrX')^* $. 
\end{theorem}
\begin{proof}
 The argument of Huybrechts works in this context almost literally, see \cite[Theorem 1.1]{LP} for the necessary changes.
\end{proof}

\subsection{Algebraically Coisotropic Subvarieties}\label{section algebraically coisotropic}
Following Voisin \cite[Definition 0.6]{Vo15}, we call a subvariety $P\subset Y$ of an irreducible symplectic manifold an \emph{algebraically coisotropic subvariety} if it is coisotropic and admits a rational map $\phi: P \ratl B$ onto a variety of dimension $\dim Y - 2\cdot\codim P$ such that the restriction of the symplectic form to $P$ satisfies $\sigma\vert_P=\phi^*\sigma_B$ for some 2-form $\sigma_B$ on $B$. In the remainder of this section, we use the deformation theoretic techniques developed so far to study algebraically coisotropic subvarieties in families.

\begin{proposition}\label{proposition cmsb}
Every irreducible component $P$ of the exceptional locus of an irreducible symplectic resolution $\pi:Y\to X$ of a projective symplectic variety $X$ is algebraically coisotropic and the coisotropic fibration of $P$ is given by the restriction $\pi\vert_P:P\to B:=\pi(P)$. In particular, it is  holomorphic. Moreover, the general fiber of $\pi\vert_P$ is rationally connected.
\end{proposition}
\begin{proof}
Let $F$ denote a resolution of singularities of the general fiber of $P\to B$. By \cite[Lemma 2.9]{Kal} it follows that the pullback of the symplectic form $\sigma$ of $Y$ to $F$ vanishes identically so that $F$ is isotropic. It remains to show that $\dim B = 2(n - \dim F)$ where $\dim Y=2n$.
 This is a consequence of \cite[Theorem 1.2]{Wierzba}. 

To prove rational connectedness, we use that for every rational curve $C$ on a symplectic variety admitting an irreducible symplectic resolution the morphism $\nu:\P^1 \to C \subset Y$ obtained by normalization and inclusion deforms in a family of dimension at least $2n-2$, see \cite[Corollary 5.1]{Ran2} or \cite[Proposition 3.1]{CP}. Being an exceptional locus, every fiber of $\pi:P\to B$ is rationally chain connected by \cite[Corollary 1.5]{HM07}. In particular, there are rational curves in $P$ that are contracted by $\pi$. Let $H$ be an ample divisor on $Y$ and take an irreducible rational curve $C$ on $P$ contracted by $\pi$ such that the intersection product $H.C$ is minimal among all such rational curves on $P$. 
In the Chow scheme of $Y$ we look at an irreducible component $\Ch$ containing $[C]$ of the locus parametrizing rational curves. 
Let $U \subset \Ch \times Y$ be the graph of the universal family of cycles. As $C$ is contracted by $\pi$, the same holds true for all curves in $\Ch$ (otherwise e.g. the intersection with a pullback of an ample divisor from $X$ would change) and as $U$ is irreducible, we have in fact $U \subset \Ch \times P$. By minimality of $H.C$ all points in $U$ correspond to irreducible and reduced rational curves. Thus, $U\to \Ch$ is a family of curves in the fibers of $P \to B$.
 
By what we noted above, $\dim \Ch\geq 2n-2$ and thus a simple dimension count and the fact that a positive dimensional family of rational curves with two basepoints has to have reducible or nonreduced members (Bend and Break, see e.g. \cite[Theorem (5.4.2)]{Ko}) imply that through the general point of a general fiber $F$ of $P\to B$ there is a family of rational curves without further basepoints of dimension $\dim(F)-1$. Consequently, $F$ is rationally connected.
\end{proof}

\begin{remark}
Rational connectedness of $F$ would follow from \cite[Theorem 9.1]{CMSB} the proof of which however seems to be incomplete. Instead, it might also be possible to use \cite[Theorem 2.8 (2)]{CMSB} and the well-known fact that a variety is rationally connected if and only if it contains a very free rational curve, see e.g. \cite[Ch IV, 3.7 Theorem]{Ko}.
\end{remark}

\begin{remark}
As mentioned in the proof of Proposition \ref{proposition cmsb}, rational chain connectedness follows from the much stronger result \cite[Corollary 1.5]{HM07}. This notion coincides for smooth varieties with rational connectedness, however, this is not the case for singular varieties. The cone over an elliptic curve is the easiest example of a variety which is rationally chain connected but not rationally connected. 
\end{remark}

Recall from \cite[Theorem 1.1]{LP2} that given an algebraically coisotropic subvariety $P$ with almost holomorphic coisotropic fibration $\phi:P \ratl B$ whose generic fiber $F$ is smooth, the subvariety $F$ deforms all over its Hodge locus $\Hdg_F \subset \Def(Y)$. Moreover, if for $t \in \Hdg_F$ we denote by $Y_t$ the corresponding deformation of $Y$, then the deformations of $F$ inside $Y_t$ cover an algebraically coisotropic subvariety $P_t \subset Y_t$ with $F_t$ as a generic fiber of the coisotropic fibration. It seems, however, unclear how to relate the cycle class of $P_t$ with that of $P$ let alone to show that $P_t$ is a flat deformation. 

In the context of birational contractions of symplectic varieties (in dimension $\geq 4$ at least) it is rather common that the generic fiber of the exceptional locus over its image is smooth. Thus, it is worthwhile to mention that the main result of \cite{LP2} can be strengthened in this special situation.
But first we need some notation.

Let $F \subset Y$ be a closed subvariety in an irreducible symplectic manifold. If $\scrY \to \Def(Y)$ denotes the universal deformation we let $\scrH \to \Def(Y)$ be the union of all those components of the relative Hilbert scheme (or Douady space) of $\scrY$ over $\Def(Y)$ which contain $[F]$. We define the closed subspace $\Def(Y,F) \subset \Def(Y)$ to be the scheme theoretic image of $\scrH \to \Def(Y)$; this is the space of deformations of $Y$ that contain a deformation of $F$. 

\begin{theorem}\label{theorem pacienza}
Let $\pi:Y\to X$ be an irreducible symplectic resolution of a projective symplectic variety $X$, let $P\subset Y$ be the exceptional locus of $\pi$, put $B:=\pi(P)$, and let $\scrY \to \scrX$ be the restriction of the universal deformation of $Y\to X$ over $\Def(Y,N)$. Suppose that $P$ is irreducible and that a general fiber $F$ of $\pi:P\to B$ is smooth. Then we have $\Def(Y,N) \subset \Def(Y,F)$ and the Hodge locus $\Hdg_P$ of $P$ contains $\Def(Y,N)$.
\end{theorem}
\begin{proof} 
Let $\scrF \to \scrH \to \Def(Y,F)$ be the universal deformation of $F$ over the closed subspace $\scrH$ of the relative Hilbert scheme of $\scrY \to \Def(Y,N)$. 
By Proposition \ref{proposition cmsb} the variety $P$ is algebraically coisotropic and the fibers of $\pi:P\to B$ are rationally connected so that $H^1(F,\sO_F)=0$ and \cite[Theorem 4.6]{LP2} can be applied. We deduce that $\Def(Y,F)$ and $\scrH\to\Def(Y,F)$ are smooth at $0$ respectively $[F]$. In particular, $\scrH$ is irreducible. Moreover, by \cite[Corollary 1.2]{LP2} the period map identifies $\Def(Y,F)$ with $Q \cap \P(K)$ where $Q \subset \P(H^2(Y,\C))$ is the period domain of the irreducible symplectic manifold $Y$ and $K=\ker(H^2(Y,\C)\to H^2(F,\C))$. If $b\in B$ denotes the point with $F=\pi^{-1}(b)$, then by commutativity of $$\xymatrix{F \ar[r]\ar[d]& Y \ar[d]\\ \{b\} \ar[r] & X\\}$$  it follows that we have $H^2(X,\C) \subset K$ and hence using $H^2(X,\C)^\perp = N$ and the period map once more we obtain $\Def(Y,N) \subset \Def(Y,F)$.

In order to show that $P$ remains a Hodge class all over $\Def(Y,N)$ we will construct flat families $\scrP_\Delta \to \Delta$ over curves $\Delta \subset \Def(Y,N)$ passing through the origin such that the cycle underlying the central fiber $\scrP_{\Delta,0}$ is a multiple of $P$. To this end we replace $\scrF \to \scrH$ as well as $\scrY \to \scrX$ by their restrictions to a given smooth curve germ $\Delta \subset \Def(Y,N)$ and obtain morphisms
$ \xymatrix{
\scrH & \ar[l] \scrF \ar[r]& \scrY \ar[r] &\scrX\\
 }$
over $\Delta$. The map in the middle is induced by the projection to the second factor of $\scrF \subset \scrH \times \scrY$. As $\Delta$ is smooth and smoothness is stable under base change we may still assume that $\scrH$, $\scrH \to \Delta$ and $\scrF \to \scrH$ are smooth at $[F]$ respectively in a neighborhood of $F\subset \scrF$. In particular, there is still a unique irreducible component of $\scrH$ passing through $[F]$ and by shrinking the representative of $\Def(Y,F)$ and throwing away components of $\scrH$ we may assume that $\scrH$ is irreducible.

On the other hand, as $\scrX \to \Delta$ is a locally trivial deformation, it induces a flat (even locally trivial) deformation of all components of its singular locus (with the reduced structure). Let $\scrB \to \Def(Y,N)$ be the so induced deformation of $B$. Then $\scrB \subset \scrX$ is an irreducible (and reduced) subspace.
If we knew that also $\scrY \to \scrX$ were a locally trivial deformation of $Y\to X$, the claim would follow immediately. As we cannot prove this so far, cf. Question \ref{question contraction deforms locally trivially}, we have to argue differently.

We take the unique closed irreducible and reduced subspace $\scrF' \subset \scrF$ that coincides with $\scrF$ in a neighborhood of $F$. As $\rho:\scrF' \to \scrH$ is proper it therefore is surjective as well. If we take the Stein factorization, $\scrF'\to \tilde\scrH \to \scrH$, then
by the Rigidity Lemma (see e.g. \cite[Lemma 1.15]{Debarre}) there is a commutative diagram
\[\xymatrix{
\scrF \ar[d]\ar[r]& \scrY\ar[d]\\
\tilde\scrH \ar[r]&\scrX\\
}\] 
Note that $\tilde\scrH \to \scrH$ is finite, birational, and an isomorphism over $[F]\in \scrH$ and that $\tilde\scrH$ is irreducible. Moreover, the image of $\tilde\scrH \to \scrX$ coincides in a neighborhood of $[f]\in \tilde\scrH$ with the closed subvariety $\scrB \subset\scrX$ thanks to the smoothness of $\scrF \to \scrH$ in a neighborhood of $F\subset \scrF$. Invoking the irreducibility of $\tilde\scrH$ and $\scrB$ we conclude that we have $\tilde\scrH\onto \scrB \subset \scrX$. We define $\scrP  = \scrP_\Delta \subset \scrY$ to be the image of $\scrF' \to \scrY$. The variety $\scrF'$ being irreducible the same holds true for $\scrP$ and hence the induced map $\rho:\scrP \to \Delta$ is flat. It remains to show that $P\subset \scrP$ is the unique component of the central fiber $\scrP_0$ of $\rho$ of dimension $\dim P$. This follows from the irreducibility of $P$ by invoking the Rigidity Lemma once more.
\end{proof}

Recall from \cite[Definition 1.5]{Vo15} that a cohomology class $p \in H^{2i}(Y,\C)$ on a symplectic manifold is called \emph{coisotropic} if it is a Hodge class and $[\sigma]^{n-i+1}\cup p = 0$ in $H^{2n+2}(Y,\C)$ where $\sigma$ is the symplectic form on $Y$. We refer to Huybrechts' article \cite[Definition 3.1]{Hu14} for the notion of a constant cycle subvariety. This is roughly speaking a subvariety of a given variety all of whose points have the same cycle class in the ambient variety.

\begin{corollary}
In the situation of Theorem \ref{theorem pacienza}, the class of $[P]$ remains an effective coisotropic Hodge class all over $S=\Def(Y,N)$. Moreover, there are varieties $P_t \subset Y_t$ for each $t\in S$ representing (a multiple of) $[P]$ which are algebraically coisotropic with rationally connected fibers. In particular, the fibers are constant cycle subvarieties of $Y_t$.
\end{corollary}
\begin{proof}
It follows from the preceding theorem that the class $[P_t]$ of the subvarieties $P_t\subset Y_t$ for a deformation $Y_t$ of $Y$ with $t \in \Def(Y,N)$ is (a multiple of) $[P]$. The proof showed moreover that $P_t$ is covered by deformations of the general fiber $F$ of the coisotropic fibration of $P$. The claim follows as $F$ was rationally connected and rational connectedness is known to be invariant under deformations (for smooth varieties).
\end{proof}

\begin{question}\label{question contraction deforms locally trivially}  Let $X$ be a projective symplectic variety, $Y\to X$ an irreducible symplectic resolution, $N=\tilde q^{-1}(N_1(Y/X)_\Q)$ (cf. Lemma \ref{lemma symplektisch}), and
$\scrY\to\scrX$ be the morphism between universal families over $S:=\Def(Y,N)=\Def^{\mathrm{lt}}(X)$ from diagram \eqref{eq diag lt}. In this case we ask:
\[
\mbox{\emph{Is $\scrY \to \scrX$ a locally trivial deformation of $Y \to X$? }}
\]
\noindent Note that if this were the case, the whole diagram 
\[\xymatrix{
\scrE \ar[d]\ar[r]& \scrY\ar[d]\\
\scrS \ar[r]&\scrX\\
}\] 
of complex spaces over $S$ where $\scrS \to S$ is the singular locus of $\scrX \to S$ and $\scrE \to S$ is the exceptional locus of $\scrY \to \scrX$ were a locally trivial (in particular flat) deformation over $S$ of its central fiber.
\end{question}

\section{Period maps and monodromy groups}\label{sec monodromy}

In this section we first show that any irreducible symplectic resolution $\pi:Y\to X$ is deformable in the sense of Definition~\ref{definition deformable}.  We then use this to develop the global theory of locally trivial deformations.

 \subsection{Cones and Period Maps in the Smooth Case}\label{sect geometric cone stuff}

This subsection contains nothing new, and serves only to summarize known results about decompositions of the positive cone for irreducible symplectic manifolds. Let $Y_0$ be a fixed irreducible symplectic manifold and denote by $\Lambda'$ the abstract lattice underlying $H^2(Y_0,\Z)$. The moduli space of $\Lambda'$-marked irreducible symplectic manifolds is the complex space $\mathfrak{M}_{\Lambda'}$ whose points are isomorphism classes of pairs $(Y,\mu)$ consisting of an irreducible symplectic manifold $Y$ deformation equivalent to $Y_0$ together with an isometry $\mu:H^2(Y,\Z)\to \Lambda'$, also referred to as a \emph{marking}. The space $\mathfrak{M}_{\Lambda'}$ obtains the structure of a not necessarily Hausdorff complex space from patching the spaces $\Def(Y)$ using the local Torelli theorem, see \cite[1.18]{Huy99}. In particular, $\gothM_{\Lambda'}$ is smooth of dimension $h^{1,1}(Y_0)=\rk (\Lambda')-2$. We also refer to  \cite[\S 4.2]{huybourbaki} and the references therein.

The \emph{period map} for $\Lambda'$-marked irreducible symplectic manifolds is defined by
\begin{equation}\label{eq period map}
P:\gothM_{\Lambda'} \to \Omega_{\Lambda'}, \quad P(Y,\mu) = \mu\left(H^{2,0}(Y)\right).
\end{equation}
Recall that in this context there is the notion of a parallel transport operator and we may consider the corresponding monodromy group $\Gamma'\subset\O(\Lambda')$, see \cite[Definition 1.1]{markmantor} for details. We continue with some preliminary remarks regarding the K\"ahler cone of an irreducible symplectic manifold.

\begin{definition}\label{definition geometric cones}
In parallel to the notation of Setup \ref{setup cones}, we denote by $C(Y)\subset H^2(Y,\R)$ and $C^{1,1}(Y)\subset H^{1,1}(Y,\R)$ the cones of positive vectors for the Beauville--Bogomolov--Fujiki form $q_Y$. Note that in the notation of Setup~\ref{setup cones} we have $C(Y)=C(H^2(Y,\Z))$ and $C^{1,1}(Y)=C^{1,1}(H)$ where $H$ denotes the Hodge structure on $H^2(Y,\Z)$. The cone  $C^{1,1}(Y)$ has two connected components, and we define $C^{1,1}(Y)^+$ to be the component containing a K\"ahler class.  
\end{definition}

For an irreducible symplectic manifold $Y$, the cone $C^{1,1}(Y)^+$ further has a wall-and-chamber decomposition whose open chambers are the images of the K\"ahler cones of birational models of $Y$ under monodromy operators that preserve the Hodge structure (see \cite[Definition 5.10]{markmantor} for details).  Following Markman, the walls and chambers of this decomposition will be referred to as \emph{K\"ahler-type walls} and \emph{K\"ahler-type chambers}, respectively. The K\"ahler-type walls of $C^{1,1}(Y)^+$ have been described by Amerik--Verbitsky in terms of monodromy birationally minimal (MBM) classes \cite[Definition 1.13]{AV} (see also the wall divisors of \cite{Mongardi}) whose definition we now recall.  

\begin{definition}\label{definition mbm classes}
A nonzero class $\alpha\in H^{1,1}(Y,\Z)$ with $q_Y(\alpha)<0$ is $MBM_Y$ if up to the action of the monodromy group $\alpha^\perp \subset H^{1,1}(Y,\R)$ is a wall of the K\"ahler cone of a birational model of $Y$.  We say a class $\alpha\in H^2(Y,\Z)$ is $MBM$ if it becomes $MBM_{Y'}$ for some deformation $Y'$ of $Y$.
\end{definition}

The following proposition summarizes what we need about MBM classes and also explains our choice of terminology in Definition~\ref{definition kaehler type chambers}. 

\begin{proposition}[Amerik-Verbitsky]\label{proposition amerik verbistky}
Let $Y$ be an irreducible symplectic manifold, denote by $\Lambda'$ the abstract lattice underlying $(H^2(Y,\Z),q_Y)$, and let $\Gamma' \subset \O(\Lambda')$ be the monodromy group. Then the following hold.
\begin{enumerate}
\item The $MBM_Y$ classes are precisely the Hodge $MBM$ classes.
	\item The K\"ahler-type walls of $C^{1,1}(Y)^+$ in the sense of Markman are precisely the hyperplanes $\alpha^\perp$ for $\alpha$ ranging over all $MBM_Y$ classes. 
	\item Assume $b_2(Y)>5$. After a choice of marking, the MBM classes form an MBM collection $MBM_{\Lambda'}\subset \Lambda'$ of $\Lambda'$ for the monodromy group $\Gamma'\subset\O(\Lambda')$ in the sense of Definition \ref{def mbm coll}. In particular, the notion of K\"ahler-type walls and chambers from Definition \ref{definition kaehler type chambers} coincides the notion due to Markman in this case.
\end{enumerate}
\end{proposition}
\begin{proof}
The first statement follows because $MBM_Y$ classes are deformation invariant along deformations for which they remain of Hodge type $(1,1)$ by \cite[Theorem 1.17]{AV} and the second is by \cite[Theorem 1.19]{AV}.  For the third, MBM classes have bounded square assuming $b_2(Y)>5$ by \cite[Theorem 5.3]{AV17}. 
\end{proof}
In particular, if $b_2(Y)>5$ the K\"ahler cone is locally polyhedral in $C^{1,1}(Y)^+$ and if $\Pic(Y)$ is negative-definite there are only finitely many K\"ahler-type chambers. Next, we summarize the relation between the wall and chamber decomposition of the positive cone and the Torelli theorem. Recall the cone space $\Omega_\Lambda^\cone$ as defined after Definition~\ref{definition kaehler type chambers}.

\begin{theorem}[Markman, Verbitsky]\label{theorem fibers of the period map}
For any connected component $\mathfrak{N}_{\Lambda'}$ of $\mathfrak{M}_{\Lambda'}$ and any $(Y,\mu)\in\mathfrak{N}_{\Lambda'}$ the points in the same fiber of $P:\mathfrak{N}_{\Lambda'}\to \Omega_{\Lambda'}$ as $(Y,\mu)$ are in natural correspondence with the K\"ahler-type chambers of $C^{1,1}(Y)^+$.

Moreover, the natural enhancement $P^\mathrm{cone}$ of the period map obtained by ``taking the K\"ahler cone"

\[\xymatrix@R-2pc{
\mathfrak{N}_{\Lambda'}\ar[r]^{P^\mathrm{cone}}& \Omega^\mathrm{cone}_{\Lambda'}\\
(Y,\mu)\ar@{|->}[r]& (P(Y,\mu),K(Y))
}\]
where $K(Y)$ is the chamber of $C^{1,1}(Y)^+$ corresponding to the K\"ahler cone of $Y$ is an isomorphism.
\end{theorem}
\begin{proof}
The map $P^\mathrm{cone}$ is continuous as K\"ahler classes remain K\"ahler in nearby deformations provided they remain of type $(1,1)$, and the fact that it is an isomorphism is obtained by combining Verbitsky's global Torelli theorem \cite[Theorem 1.17]{V13} and Markman's results on the K\"ahler-type chamber decomposition \cite[Theorem 5.16]{markmantor}.
\end{proof}

\subsection{Deformability of resolutions}\label{section deformability}

Let $\pi:Y \to X$ be an irreducible symplectic resolution. Here we show that $\pi$ is deformable. We will need the following

\begin{lemma}\label{lemma mbm}Let $\pi:Y\to X$ be an irreducible symplectic resolution with $b_2(Y)> 5$.  Then $N(Y/X):=(\pi^*H^2(X,\Z))^\perp \subset H^2(Y,\Z)$ is rationally generated by $MBM_Y$ classes.
\end{lemma}
\begin{proof}  
We know that that $N(Y/X)^\perp$ in $H^{1,1}(Y,\R)$ intersects the nef cone of $Y$ in an extremal face $\tau$ because $N(Y/X)$ is generated by the duals of curves that are contracted. As $\tau$ contains a positive class $\omega$ (namely the pullback of a K\"ahler class of $X$) and the nef cone is polyhedral near $\omega$ (by Remark \ref{remark MBM gen}) and cut out by $\alpha^\perp$ for $\alpha$ ranging over all $MBM_Y$ classes, it follows that $N(Y/X)_\Q$ is generated by $MBM_Y$ classes.   
\end{proof}

We are thus in the situation of Setup \ref{setup mbm sublattice} where we can ``restrict" the cone decomposition of $Y$ to $X$.  Precisely, given an irreducible symplectic resolution $\pi:Y\to X$ we denote by $H$ the Hodge structure on $H^2(X,\Z)$ and set
\[
C^{1,1}(X)^+:= C^{1,1}(H)^+
\]
in analogy with Definition~\ref{definition geometric cones}. We let $\Lambda'$ be the lattice underlying $H^2(Y,\Z)$, $M'$ be the MBM classes in $\Lambda'$,  and $\Lambda$ be the sublattice underlying $\pi^*H^2(X,\Z)$. In the notation of Setup~\ref{setup mbm sublattice} we then have that $N$ is given by $\left(\tilde q_Y\right)^{-1}(N_1(\tX/X)_\Q)$ for $\tilde q_Y$ as in Lemma~\ref{lemma symplektisch}, and the role of the cones $C^{1,1}(H_p')^+$ and $C^{1,1}(H_p)^+$ is played by $C^{1,1}(Y)^+$ and $C^{1,1}(X)^+$, respectively. 
We thereby refer to the resolution chambers of $C^{1,1}(Y)^+$ with respect to $\pi$ in the sense of Definition \ref{definition resolution chambers}. We also obtain a decomposition of $C^{1,1}(X)^+$ into K\"ahler type chambers. This will be needed in the proof of the following:

\begin{proposition}\label{proposition nonproj deform}Every irreducible symplectic resolution $\pi:Y\to X$ with $b_2(X)>4$ is deformable.
\end{proposition}
\begin{proof}  
Note that if $b_2(Y)=b_2(X)$ then $\pi$ is an isomorphism by \eqref{eq relative homology sequence} and there is nothing to prove.  We may thus assume $b_2(Y)>5$.  Choose a marking $\mu:\Lambda'\xrightarrow{\cong} H^2(Y,\Z)$ and define $\Lambda:=\mu^{-1}(\pi^*H^2(X,\Z))$.   Let $\mathfrak{N}_{\Lambda'}$ be the component of $\mathfrak{M}_{\Lambda'}$ containing $(Y,\mu)$.  Consider the subspace $\mathfrak{N}_{\Lambda',\Lambda}$ of $\mathfrak{N}_{\Lambda'}$ for which $N=\Lambda^\perp$ is Hodge, i.e. $\mathfrak{N}_{\Lambda',\Lambda}$ is the preimage of $\Omega_{\Lambda}\subset \Omega_{\Lambda'}$ under the period map \eqref{eq period map} restricted to $\gothN_{\Lambda'}$.  Let $\Gamma_N\subset\O(\Lambda')$ be the subgroup of the monodromy group of $Y$ which fixes $N$ pointwise.  For a choice of resolution chamber $\tau$ of $\pi$ (which we identify with its image in $\Lambda'$), let $\mathfrak{N}_{\Lambda',\Lambda}^\tau\subset\mathfrak{N}_{\Lambda',\Lambda}$ be the subspace of those $(Y',\mu')$ whose K\"ahler cone is contained in $\tau$ and has $\Lambda$ as a face.  By Theorem \ref{theorem fibers of the period map}, the map $\mathfrak{N}_{\Lambda'}\to \Omega_{\Lambda'}^\mathrm{cone}$ is an isomorphism. Evidently it restricts to an isomorphism $\mathfrak{N}^\tau_{\Lambda',\Lambda}\to \Omega^\mathrm{cone}_{\Lambda',\Lambda}(\tau)$ in the notation of Definition \ref{definition resolution chambers}.  By Proposition~\ref{prop face map}, the ``face" map
\[
\mathfrak{N}_{\Lambda',\Lambda}^\tau\to \Omega_\Lambda^\mathrm{cone}
\]
is an isomorphism.  Clearly $\Gamma_N$ acts on both $\mathfrak{N}_{\Lambda',\Lambda}^\tau$ and $\Omega_\Lambda^\mathrm{cone} $ in a compatible way. 

By Proposition \ref{proposition deformation remains isv} we may assume $X$ is not projective, and in particular not of maximal Picard rank.  Now, $\Gamma_N$ is a finite index subgroup of $\O(\Lambda)$, so by Proposition \ref{proposition verbitsky} and Corollary~\ref{corollary lift density} we can choose a projective point $(Y_0,\mu_0)\in \mathfrak{N}_{\Lambda',\Lambda}^\tau$ in the orbit closure of $(Y,\mu)$.  By the basepoint free theorem \cite[Theorem 6.1]{Ka85} there is a contraction \mbox{$\pi_0:Y_0\to X_0$} to a primitive symplectic variety contracting the classes in $N$, and moreover by Proposition \ref{proposition deformation remains isv}  this contraction deforms to $\mathscr{Y}_0\to\mathscr{X}_0$ over $\Def(Y_0,N)\subset \mathfrak{N}_{\Lambda',\Lambda}^\tau$, with $\mathscr{X}_0$ deforming locally trivially.  After changing the marking by an element of $\Gamma_N$, we have $(Y,\mu')\in \Def(Y_0,N)$, and therefore there is a deformable contraction $\pi':Y\to X'$.  As $\pi'$ contracts exactly $N$, it follows that this is in fact the original contraction $\pi:Y\to X$.  
\end{proof}

We record the following corollary of the proof, which is was pointed out by Amerik--Verbitsky \cite[Theorem 5.5]{AV19}.  The condition on the dimension of the face is required to ensure the contraction $X$ has $b_2(X)>4$.
\begin{corollary}\label{corollary can contract}Let $Y$ be a smooth irreducible symplectic manifold and $\tau\subset H^{1,1}(Y,\R)$ a face of the K\"ahler cone meeting the positive cone $C^{1,1}(Y)^+$ for which $\dim\tau>2$.  Then there is a birational contraction $\pi:Y\to X$ contracting precisely $\tau^\perp$.
\end{corollary}

\begin{remark}\label{remark rigid}
It should be noted that the assumption  $b_2(X)>4$ that we used several times throughout this section is nontrivial. Unlike in the smooth case, we know that there are examples of symplectic varieties $X$ with $b_2(X)=3$ {that admit an irreducible symplectic resolution}:  it may well happen that a birational contraction $Y \to X$ of an irreducible symplectic manifold $Y$ contracts a negative definite subspace of $H^2(Y,\R)$ of maximal dimension. In the case of $K3$ surfaces for example, one may take $X$ to be $S/G$ where $S$ is a $K3$ surface and $G$ is a group of symplectic automorphisms with minimal invariant second cohomology lattice $H^2(S,\Z)^G$ (i.e. of rank $3$).
 Finite groups of symplectic automorphisms were classified by Nikulin \cite{Nikulin} and Mukai \cite{Mukai}, explicit examples of groups with the sought-for rank of $H^2(S,\Z)^G$ may be found in \cite{Xi,Ha}. An example of a different kind may be found in \cite{OZ}. The minimal resolution $Y \to X$ then gives us an example of a contraction of relative Picard rank $19$. The induced contraction $\Hilb^n(Y) \to \Sym^n(Y) \to \Sym^n(X)$ gives an example of a contraction of a \Kthreen-type variety of dimension $2n$ with relative Picard rank $20$.
\end{remark}

\subsection{Marked moduli spaces -- the singular case}  

Recall that given a lattice $\Lambda$ with quadratic form $q$ of signature $(3,\rk(\Lambda)-3)$, we define an analytic coarse moduli space $\mathfrak{M}_{\Lambda}$ of $\Lambda$-marked irreducible symplectic manifolds by gluing together the Kuranishi spaces.  Likewise, we define $\mathfrak{M}^\mathrm{lt}_\Lambda$ to be the analytic space obtained by gluing together the locally trivially Kuranishi spaces of $\Lambda$-marked symplectic varieties admitting irreducible symplectic resolutions. By Theorem \ref{theorem deflt is smooth}, $\mathfrak{M}^\mathrm{lt}_\Lambda$ is a not-necessarily-Hausdorff complex manifold.

With the period domain $\Omega_\Lambda$ defined as in \eqref{eq hk period domain}, by the local Torelli theorem, see Corollary \ref{proposition local torelli}, there is a period map $P:\mathfrak{M}^{\mathrm{lt}}_\Lambda\to \Omega_\Lambda$ that is a local isomorphism.

Let $(X,\nu)$ be a $\Lambda$-marked symplectic variety admitting an irreducible symplectic resolution.  By the Proposition \ref{proposition nonproj deform} and \cite[Theorem 2.5]{Huy}, the deformation type of an irreducible symplectic resolution is constant along each connected component of $\mathfrak{M}^\mathrm{lt}_\Lambda$.  Given a lattice $\Lambda'$ with quadratic form $q'$ of signature $(3,\rk(\Lambda')-3)$ and a primitive embedding of lattices $\iota:\Lambda\into\Lambda'$, we define a \emph{compatibly marked irreducible symplectic resolution} $\pi:(Y,\mu)\to (X,\nu)$ to be an irreducible symplectic resolution $\pi$ and a commutative diagram

\[\xymatrix{
\Lambda'\ar[r]^{\mu}_\cong&H^2(Y,\Z)\\
\Lambda\ar[u]^\iota\ar[r]_\nu^\cong&H^2(X,\Z)\ar[u]_{\pi^*}
}\]
compatible with $q$ and $q'$. We define $\mathfrak{M}^\mathrm{res}_{\Lambda',\Lambda}$ to be the set of compatibly marked irreducible symplectic resolutions $(Y,\mu)\to(X,\nu)$ modulo the following equivalence relation: we identify $\pi:(Y,\mu) \to (X,\nu)$ with $\pi':(Y',\mu') \to (X',\nu')$ provided there is an isomorphism 
\[
\xymatrix{Y\ar[r]^\cong\ar[d]_{\pi}&Y'\ar[d]^{\pi'}\\
X \ar[r]_\cong&X'}
\]
 compatible with the markings.  There are obvious forgetful maps that fit into a diagram
 \begin{equation}\label{eq moduli spaces}
\xymatrix{
&\mathfrak{M}^\mathrm{lt}_\Lambda\ar[r]^P&\Omega_{\Lambda}\ar[dd]\\ 
\mathfrak{M}^\res_{\Lambda',\Lambda}\ar[ur]\ar[dr]&&\\
&\mathfrak{M}_{\Lambda'}\ar[r]_P&\Omega_{\Lambda'}\\
 }
\end{equation}
 where the vertical arrow is the embedding of $\Omega_\Lambda$ into $\Omega_{\Lambda'}$ as the Noether-Lefschetz locus $\Omega_{\Lambda'}\cap\P(\Lambda\otimes\C)$. The set $\mathfrak{M}^\res_{\Lambda',\Lambda}$ is given the structure of an analytic space by requiring the top diagonal map to be a local isomorphism, using Proposition \ref{prop defo}.

For the following proposition, let $\mathfrak{M}_{\Lambda',\Lambda}$ be the inverse image under $P$ of $\Omega_{\Lambda}\subset\Omega_{\Lambda'}$ in $\mathfrak{M}_{\Lambda'}$---that is, the locus where $\Lambda^\perp$ is Hodge.  Recall from the proof of Proposition~\ref{proposition nonproj deform} that for a component $\mathfrak{N}_{\Lambda'}\subset\mathfrak{M}_{\Lambda'}$ and for a choice of resolution chamber $\tau$ the subset $\mathfrak{N}^\tau_{\Lambda',\Lambda}$ consists of those $(Y,\mu)$ for which the K\"ahler cone of $Y$ is contained in $\tau$ and has $\Lambda_\R$ as a face.  We have $\mathfrak{N}^\tau_{\Lambda',\Lambda}\cong \Omega^\mathrm{cone}_{\Lambda',\Lambda}(\tau)$ in the notation of Section \ref{sect cone stuff}.

\begin{proposition}\label{proposition meta monodromy}Assume $\rk(\Lambda)>4$.  For each choice of component $\mathfrak{N}^\mathrm{res}_{\Lambda',\Lambda}$ of $\mathfrak{M}_{\Lambda',\Lambda}^\res$, the diagonal map from \eqref{eq moduli spaces} yields an isomorphism onto some $\mathfrak{N}^\tau_{\Lambda',\Lambda}\subset \mathfrak{M}_{\Lambda',\Lambda}$ (respectively some component $\mathfrak{N}_{\Lambda'}^\lt$ of $\mathfrak{M}_\Lambda^\lt$).  Moreover, each such $\mathfrak{N}^\tau_{\Lambda',\Lambda}$ (respectively $\mathfrak{N}_\Lambda^\lt$) arises as the image of some component $\mathfrak{N}^\mathrm{res}_{\Lambda',\Lambda}$ of $\mathfrak{M}_{\Lambda',\Lambda}^\res$.
\end{proposition}
\begin{proof}  Given the first claim, the second claim is obvious by Corollary \ref{corollary can contract} (respectively by taking a resolution).  Both diagonal maps are local isomorphisms, so for the first claim we need only show bijectivity.

For the bottom diagonal map, $\mathfrak{N}^\mathrm{res}_{\Lambda',\Lambda}$ lands in the union $\bigcup_\tau\mathfrak{N}^\tau_{\Lambda',\Lambda}$, and therefore in a single $\mathfrak{N}^\tau_{\Lambda',\Lambda}$.  Corollary \ref{corollary can contract} implies that any path in $\mathfrak{N}^\tau_{\Lambda',\Lambda}$ can be locally lifted to $\mathfrak{N}^\mathrm{res}_{\Lambda',\Lambda}$, as both are locally isomorphic to $\Omega_\Lambda$.  The lift of a point is unique if it exists (as a contraction is unique if it exists and the marking is determined), so paths lift globally, and it follows that $\mathfrak{N}^\mathrm{res}_{\Lambda',\Lambda}$ bijects onto $\mathfrak{N}^\tau_{\Lambda',\Lambda}$.

For the top diagonal map, suppose $\mathfrak{N}^\mathrm{res}_{\Lambda',\Lambda}$ lands in a component $\mathfrak{N}^\lt_{\Lambda}$.  Note that a very general point in $\mathfrak{M}^\lt_\Lambda$ has uniformly finitely many lifts to $\mathfrak{M}^\mathrm{res}_{\Lambda',\Lambda}$, as the K\"ahler cone decomposition of a resolution $Y$ has finitely many chambers and the marking extends in finitely many ways.  By Proposition \ref{proposition nonproj deform}, for any path $\gamma\subset\mathfrak{N}^\lt_{\Lambda}$ and for each very general point $x\in\gamma$ we can find an open neighborhood in which the path lifts through every lift of $x$.  As $\gamma$ can be covered by finitely many such neighborhoods, paths can be lifted globally, and $\mathfrak{N}^\mathrm{res}_{\Lambda',\Lambda}\to\mathfrak{N}^\lt_{\Lambda}$ is surjective.

The injectivity claim of the top diagonal map means the following:  if we have two resolutions $\pi:Y\to X$ and $\pi':Y'\to X$ together with a parallel transport operator $f:H^2(Y,\Z)\to H^2(Y',\Z)$ arising from a simultaneously resolved family connecting $\pi$ to $\pi'$ which is the identity on $H^2(X,\Z)$, then $f$ arises from an isomorphism $\phi:Y\to Y'$ such that $\pi'\circ\phi=\pi$.  By the statement of the proposition regarding the lower diagonal map, $f$ must send the resolution chamber $\tau$ of $\pi$ containing the K\"ahler cone of $Y$ to the resolution chamber $\tau'$ of $\pi'$ containing the K\"ahler cone of $Y'$.  Moreover, $f$ fixes a K\"ahler class on $X$ and therefore both the image of the nef cone of $Y$ and the nef cone of $Y'$ have the same intersection with $\pi'^*H^2(X,\R)$.  It follows from Proposition \ref{prop face map} that $f$ maps the K\"ahler cone of $Y$ to the K\"ahler cone of $Y'$ and the claim then follows.
\end{proof}

\begin{corollary}\label{cor birat}Assume $\rk(\Lambda)>4$.  If $(X,\nu)$ and $(X',\nu')$ are inseparable in $\mathfrak{M}^\lt_\Lambda$, then $X$ and $X'$ are birational.
\end{corollary}
\begin{proof}$(X,\nu)$ and $(X',\nu')$ admit marked resolutions that are inseperable in moduli, hence birational. 
\end{proof}

\subsection{Monodromy groups}
 
 Each of the above moduli spaces $\mathfrak{M}$ has a corresponding notion of parallel transport operator, which we call \emph{locally trivial parallel transport operators} for $\mathfrak{M}^\mathrm{lt}_\Lambda$ and \emph{simultaneously resolved parallel transport operators} for $\mathfrak{M}^\res_{\Lambda',\Lambda}$. Precisely, for $(X,\nu)$ and $(X',\nu')$ in the same component of $\mathfrak{M}^\lt_\Lambda$, we define $\nu'\circ\nu^{-1}:H^2(X,\Z)\to H^2(X',\Z)$ to be a locally trivial parallel transport operator, and likewise for the simultaneously resolved parallel transport operators.  A simultaneously resolved parallel transport operator from $\pi:Y\to X$ to $\pi':Y'\to X'$ yields a diagram
 \[\xymatrix{
 H^2(Y,\Z)\ar[r]^f&H^2(Y',\Z)\\
 H^2(X,\Z)\ar[u]^{\pi^*}\ar[r]_g&H^2(X',\Z)\ar[u]_{\pi'^*}
 }\]
 and $f,g$ are evidently locally trivial parallel transport operators.
 
 First note the following:
 
\begin{corollary} Assume $\rk(\Lambda)>4$.  Any locally trivial (respectively simultaneously resolved) parallel transport operator can be realized as parallel transport along a path in a locally trivial (respectively simultaneously resolved) family over a connected analytic base.
\end{corollary}
\begin{proof}It follows from the Proposition \ref{proposition meta monodromy} that a Picard rank zero point of $\mathfrak{M}^\lt_\Lambda$ (respectively $\mathfrak{M}^\mathrm{res}_{\Lambda',\Lambda}$) is a separated point (i.e. can be separated from any other point).  Thus, any path in $\mathfrak{M}^\lt_\Lambda$ (respectively $\mathfrak{M}^\mathrm{res}_{\Lambda',\Lambda}$) can be covered by finitely many open sets over which families exist, and these may be glued at very general points.
\end{proof}

 Proposition \ref{proposition meta monodromy} allows us to describe the relationship between these three notions of parallel transport operators.
 \begin{corollary}\label{cor parallel transport}Suppose $X$ admits an irreducible symplectic resolution and \mbox{$b_2(X)>4$}.
 \begin{enumerate}
  \item For each choice of irreducible symplectic resolution $\pi:Y\to X$, a locally trivial parallel transport operator $g:H^2(X,\Z)\to H^2(X',\Z)$ lifts uniquely to a simultaneously resolved parallel transport operator. 
 \item Let $\pi:Y\to X$ be an irreducible symplectic resolution.  A parallel transport operator $f:H^2(Y,\Z)\to H^2(Y',\Z)$ for which $f(N(Y/X))$ is Hodge extends to a simultaneously resolved parallel transport operator if and only if the image of the resolution K\"ahler cone of $Y$ contains the K\"ahler cone of $Y'$.

 \end{enumerate}
 \end{corollary}
 
Let us now turn to the associated monodromy groups.  We adopt the following notation:

\begin{enumerate}
 \item 
$\Mon^2(X)^\mathrm{lt}\subset\O(H^2(X,\Z))$ will be the image of the monodromy representation associated to locally trivial families. Here, the orthogonal group is taken with respect to the Beauville--Bogomolov--Fujiki form $q_X$ on $H^2(X,\C)$, see Remark~\ref{singular BBF}.  Likewise we define $\Mon^2(Y)$ if $Y$ is a smooth irreducible symplectic manifold.
 \item 
 Likewise for an irreducible symplectic resolution $\pi:Y\to X$ we define $\Mon^2(\pi)\subset \O(H^2(Y,\Z))\times\O(H^2(X,\Z))$ to be the monodromy group associated to simultaneously resolved families.  Note that we can in fact think of $\Mon^2(\pi)\subset \O(H^2(Y,\Z))$.
\end{enumerate}
 
 Let $\pi:Y\to X$ be an irreducible symplectic resolution.  The monodromy group $\Mon^2(Y)$ clearly acts on the resolution chambers of $\pi$.
\begin{corollary}Let $\pi:Y\to X$ be an irreducible symplectic resolution with $b_2(X)>4$.
\begin{enumerate}
\item $\Mon^2(\pi)\subset\Mon^2(Y)$ is the stabilizer of the resolution chamber of $\pi$ containing the K\"ahler cone of $Y$.
\item $\Mon^2(X)^\lt$ is the image of $\Mon^2(\pi)$ in $\O(H^2(X,\Z))$.  
\end{enumerate}
\end{corollary}      
 \begin{corollary}\label{cor finite index}  With the above setup, $\Mon^2(\pi)$ (respectively $\Mon^2(X)^\lt$) has finite index in $\O(H^2(Y,\Z))$ (respectively $\O(H^2(X,\Z))$).
 \end{corollary}
 
\subsection{Global Torelli theorem}
As noted after Lemma \ref{lemma mbm}, under the assumption on the second Betti number, an irreducible symplectic resolution determines an induced MBM collection. 
\begin{definition}\label{def mbm sing}Let $\pi:Y\to X$ be an irreducible symplectic resolution with $b_2(Y)>5$.  We define the $MBM$ classes of $X$ (with respect to $\pi$) as the $MBM$ collection of $H^2(X,\Z)$ induced by the $MBM$ classes of $H^2(Y,\Z)$ (via the embedding $\pi^*:H^2(X,\Z)\to H^2(Y,\Z)$) in the sense of Definition \ref{def induced mbm} (and using Lemma \ref{lemma mbm}).  We say a class $\alpha\in H^2(X,\Z)$ is $MBM_X$ if it is in the $MBM$ collection and Hodge.
\end{definition}
As in Section \ref{sect cone stuff}, the $MBM$ classes of $H^2(X,\Z)$ then define an induced wall-and-chamber decomposition of $C^{1,1}(X)^+$ whose open chambers are the connected components of 
\begin{equation}C^{1,1}(X)^+-\bigcup_{\alpha\in MBM_X}\alpha^\perp\label{eq induced cone}\end{equation}
which via $\pi^*$ is identified with
$$\pi^*C^{1,1}(X)^+-\bigcup_{\alpha'\in MBM_Y-N(Y/X)}\alpha'^\perp.$$  As in Definition \ref{definition kaehler type chambers} we refer to this as the K\"ahler-type wall-and-chamber decomposition.

\begin{proposition}\label{proposition torelli cone}Let $X$ be a symplectic variety admitting an irreducible symplectic resolution with $b_2(X)>4$.  The $MBM$ classes of $X$ do not depend on a choice of resolution and are locally trivial deformation invariant.  Moreover, for any component $\mathfrak{N}^\lt_\Lambda$ of the locally trivial marked moduli space of $X$, the natural map
\[P^\mathrm{cone}:\mathfrak{N}^\lt_\Lambda\to\Omega_\Lambda^\mathrm{cone}\]
by taking the K\"ahler-type chamber containing the K\"ahler cone is an isomorphism, where we define $\Omega_\Lambda^\mathrm{cone}$ as in Section \ref{sect cone stuff} with respect to the $MBM$ classes.\end{proposition}
Note that it is not clear that the K\"ahler cone of $X$ is an open chamber of the K\"ahler-type decomposition.
\begin{proof}
Picking an irreducible symplectic resolution $\pi:Y\to X$ and a resolution chamber $\tau$ of $\pi$ we obtain an identification $\mathfrak{N}^\tau_{\Lambda',\Lambda}\to\Omega^\mathrm{cone}_\Lambda$ as in the proof of Proposition \ref{proposition nonproj deform}, where $\Omega^\mathrm{cone}_\Lambda$ is defined using the $MBM$ classes of $X$ with respect to $\pi$.  From Proposition \ref{proposition meta monodromy} it then follows that $\mathfrak{N}^\lt_\Lambda\to\Omega_\Lambda^\mathrm{cone}$ is an isomorphism.  Noting that a class $\alpha\in \Lambda$ is $MBM$ (with respect to $\pi$) if and only if the fiber of $\Omega^\mathrm{cone}_\Lambda\to\Omega_\Lambda$ above a period with Picard group generated by $\alpha$ has exactly two points, it follows that the notion is independent of the resolution and locally trivial deformation invariant.
\end{proof}

We have now proved all the statements in our main Torelli theorem.

\begin{proof}[Proof of Theorem \ref{intro theorem torelli}]  Parts (1) and (2) follow from the more precise formulation in Proposition \ref{proposition torelli cone} and Corollary \ref{cor birat}.  Part (3) is Corollary \ref{cor finite index}.
\end{proof}

\begin{remark}\label{remark ghs}
The Hausdorff reduction of the moduli space of quasi-polarized locally trivial deformations of a given projective symplectic variety admitting an irreducible symplectic resolution is thus a locally symmetric variety of orthogonal type. Therefore, Mumford's theory of toroidal compactifications applies and enables one to use methods from algebraic geometry. The geometry of such varieties, and in particular their Kodaira dimensions, have been studied using modular forms by Gritsenko--Hulek--Sankaran in a series of papers, see for example \cite{GHS,GHS2,GHS3}. A rather general recent result in this direction has been obtained by Ma \cite{Ma17}. The singularities of compactifications of such moduli spaces have studied by Giovenzana in \cite{Luca}.
\end{remark}

\subsection{On Namikawa's result}
We conclude this section by generalizing the result of Namikawa \cite[Theorem (2.2)]{Na01}, used in the proof of Proposition \ref{proposition deformation remains isv}.
  
 \begin{proposition}\label{prop namikawa general}Let $\pi:Y\to X$ be an irreducible symplectic resolution with $b_2(X)>4$.  Then $\Def(X)$ is smooth of the same dimension as $\Def(Y)$ and the induced map $p:\Def(Y)\to\Def(X)$ is finite.
  \end{proposition}
  \begin{proof}
  Consider the analytic locus $Z\subset\Def^{\lt}(X)$ where the dimension of the tangent space $\dim\Ext^1(\Omega^1_{\mathscr{X}_t},\mathcal{O}_{\mathscr{X}_t})$ to the full Kuranishi space $\Def(\scrX_t)$ is nongeneric, that is, stirctly greater than $\dim H^{1,1}(Y)$.  The projective points in $\Def^\lt(X)$ (and in every subvariety) are dense by Corollary \ref{corollary projective new}, so by \cite[Theorem (2.2)]{Na01} we have that $Z$ is at most the distinguished point $0\in\Def^\lt(X)$ (possibly after shrinking $\Def^\lt(X)$).  Now we may assume $X$ is not projective (and in particular not of maximal Picard rank), so for any open neighborhood $0\in U\subset\Def^\lt(X)$, by Proposition \ref{proposition verbitsky}, Corollary \ref{corollary lift density}, and Proposition \ref{proposition torelli cone} there is another point $0\neq t\in U$ corresponding to an isomorphic variety $\mathscr{X}_t\cong X$, and thus $\dim\Ext^1(\Omega^1_{{X}},\mathcal{O}_{{X}})=\dim H^{1,1}(Y)$.

We now claim $p:\Def(Y)\to \Def(X)$ has no positive-dimensional fibers.  If it did, then the projective points would be dense in such a fiber by \cite[Proposition 26.6]{GHJ}, and by the same argument as in Corollary  \ref{corollary projective new} therefore $X$ would be projective, contradicting Namikawa's theorem \cite[Theorem (2.2)]{Na01}. Thus, $$\dim H^{1,1}(Y)  = \dim\Def(Y) \leq \dim \Def(X) \leq \dim\Ext^1(\Omega^1_{{X}},\mathcal{O}_{{X}})$$ and by the above we have equality everywhere. This proves that $\Def(X)$  must be smooth at $0$ as well. The finiteness claim follows from quasi-finiteness upon shrinking the representative $\Def(X)$ as in Namikawa's proof from \cite[3.2 Lemma, p. 132]{Fi87}. 
\end{proof}

\section{Applications to \texorpdfstring{\Kthreen}{K3\textasciicircum {[n]}}-type manifolds}\label{sec k3type}
Recall that a compact K\"ahler manifold $Y$ is said to be of \Kthreen-type if it is deformation equivalent to a Hilbert scheme of $n$ points on a \Kthree surface.  The \Kthreen-type manifolds form one of the two known infinite families of irreducible symplectic manifolds.  We assume throughout that $n\geq 2$ (i.e. that $\dim Y\geq 4$).   

By work of Markman \cite[Corollary 9.5]{markmantor} there is a canonical extension of weight 2 integral Hodge structures
\begin{equation}
0\to H^2(Y,\Z)\to \tilde \Lambda (Y,\Z)\to Q\to 0\label{extperiod}
\end{equation}
where $Q\cong \Z(-1)$.  The lattice underlying $\tilde\Lambda(Y,\Z)$ is the Mukai lattice $\tilde\Lambda_{K3}=E_8(-1)^2\oplus U^4$.  We denote the primitive generator of the orthogonal to $H^2(Y,\Z)$ in $\tilde\Lambda(Y,\Z)$ by $v=v(Y)$, which is determined up to sign and satisfies $v^2=2-2n$.  Note that
\[H^2(Y,\Z)\cong \Lambda_{K3^{[n]}}:=E_8(-1)^2\oplus U^3\oplus (2-2n).\]

Denote by $\Mon^2(K3^{[n]})\subset O(\Lambda_{K3^{[n]}})$ the image of the weight two monodromy representation, which has been computed by Markman \cite{markmanmon} to be the subgroup $\tilde O^+(\Lambda_{K3^{[n]}})$ preserving the orientation class and acting as $\pm 1$ on the discriminant group $D(\Lambda_{K3^{[n]}}):=\Lambda_{K3^{[n]}}^*/\Lambda_{K3^{[n]}}$.  We have the following well-known consequence of this computation and Verbitsky's global Torelli theorem:  the extension \eqref{extperiod} determines the birational class of $Y$.  More precisely, for two symplectic manifolds $Y,Y'$, there is a Hodge isometry $\phi:H^2(Y,\Z)\to H^2(Y',\Z)$ lifting to a Hodge isometry $\tilde\phi:\tilde\Lambda(Y,\Z)\to\tilde\Lambda(Y',\Z)$ of the Markman Hodge structures if and only if $Y$ is birational to $Y'$ \cite[Corollary 9.8]{markmantor}.  We therefore refer to \eqref{extperiod} as the \emph{extended period}.

\subsection{Bridgeland stability conditions}\label{subsec bridgeland stability}

Recall that for a \Kthree surface $S$, the total cohomology $H^*(S,\Z)$ carries the so-called Mukai Hodge structure
\[\tilde H(S,\Z):= H^0(S,\Z)(-1)\oplus H^2(S,\Z)\oplus H^4(S,\Z)(1)\]
which comes equipped with the Mukai pairing defined by
\[(a_0+a_2+a_4,b_0+b_2+b_4):=(a_2,b_2)_S - a_0b_4 - a_4b_0
\]
for $a_i, b_i \in H^i(S,\Z)$.  The Bridgeland stability condition and moduli space formalism we discuss below will also work in the larger category of twisted \Kthree surfaces.  Recall that the cohomological Brauer group $\Br(S)$ of a \Kthree surface $S$ can defined as the torsion part of the cohomology group $H^2(S,\sO_S^*)$ in the analytic topology and is naturally the image of $H^2(S,\Q/\Z)$ under the exponential map $e(-)=e^{2\pi i-}$.  A twisted \Kthree surface $(S,\alpha)$ in the sense of \cite[\S 1]{HS} is a \Kthree surface together with a Brauer class $\alpha\in \Br(S)$ and a choice of $\beta\in H^2(S,\Q)$ with $e(\beta)=\alpha$, which exists for any $\alpha$ since $H^3(S,\Z)=0$.  A twisted \Kthree surface $(S,\alpha)$ likewise has a Mukai Hodge structure $\tilde H(S,\alpha,\Z)$ with underlying lattice $H^2(S,\Z)$, and we say $(S,\alpha)$ is projective if $S$ is.  We refer to \cite[\S2]{BM} and the references therein for more details and for the theory of $\alpha$-twisted sheaves and Bridgeland stability conditions on twisted \Kthree surfaces.

Throughout the following we will only consider primitive Mukai vectors $v\in \tilde H(S,\alpha,\Z)$.  By work of Bayer--Macr\`i \cite[Theorem 1.3]{BM1}, for a generic Bridgeland stability condition $\sigma$ on $(S,\alpha)$, the moduli space $Y=M_\sigma(v)$ of Bridgeland $\sigma$-stable objects on $(S,\alpha)$ of a Mukai vector $v\in \tilde H(S,\alpha,\Z)_{\textrm{alg}}$ is a projective \Kthreen-type manifold, and we canonically have  $\tilde\Lambda(Y,\Z)=\tilde H(S,\alpha,\Z) $ with $v(Y)=v$.  The identification $v^\perp\xrightarrow{\cong} H^2(Y,\Z)$ is achieved by the Fourier--Mukai transform.

Note that by \cite[Theorem 1.2(c)]{BM}, every symplectic birational model of a Bridgeland moduli space is a Bridgeland moduli space. We will need below the following Hodge-theoretic characterization of Bridgeland moduli spaces which follows from \cite[Lemma 2.5]{Hu15} and \cite[Proposition~4]{Addington}.  
\begin{proposition}\label{proposition criterion}
A projective \Kthreen-type manifold $Y$ is isomorphic to a Bridgeland moduli space on a projective twisted \Kthree surface if and only if one of the following equivalent conditions holds:
\begin{enumerate}
\item $\tilde\Lambda(Y,\Q)_{\mathrm{alg}}$ contains $U$ as a sublattice;
\item The rational transcendental lattice $H^2(Y,\Q)_{\mathrm{tr}}\cong \tilde\Lambda(Y,\Q)_{\mathrm{tr}}$ is Hodge-isometric to the rational transcendental lattice of a projective \Kthree surface.
\end{enumerate}
Furthermore, the projective twisted \Kthree surface can be taken untwisted if and only if one of the following equivalent conditions holds: 
\begin{enumerate}
\item[$(1')$] $\tilde\Lambda(Y,\Z)_{\mathrm{alg}}$ contains $U$ as a primitive sublattice;
\item[$(2')$] The transcendental lattice $H^2(Y,\Z)_{\mathrm{tr}}\cong \tilde\Lambda(Y,\Z)_{\mathrm{tr}}$ is Hodge-isometric to the transcendental lattice of a projective \Kthree surface.
\end{enumerate}
\end{proposition}

\subsection{\Kthreen-type contractions}
We now turn to the singular case.
\begin{definition}Let $X$ be a symplectic variety and $\pi:Y\to X$ a irreducible symplectic resolution.  We say $\pi$ is a \Kthreen-type contraction if in addition $Y$ is a \Kthreen-type manifold.  We will often abuse terminology and refer to $X$ itself as a \Kthreen-type contraction as well.
\end{definition}
Note that if $X$ is a \Kthreen-type contraction, \emph{every} (K\"ahler) symplectic resolution is a \Kthreen-type manifold by Huybrechts' theorem \cite[Theorem 2.5]{Huy}.
\begin{example}Our main source of examples of \Kthreen-type contractions come from contractions of Bridgeland moduli spaces, which we call Bridgeland contractions.  Bridgeland moduli spaces of untwisted $K3$ surfaces will be called untwisted Bridgeland contractions for emphasis.  Their geometry is beautifully described via wall-crossing by Bayer--Macr\`i theory \cite{BM}. Given a projective twisted $K3$ surface $(S,\alpha)$, a primitive Mukai vector $v\in \tilde H(S,\alpha,\Z)_{\mathrm{alg}}$, and an open chamber $\mathcal{C}\subset\Stab^\dagger(S)$ associated to $v$, then by \cite[Theorem 1.4(a)]{BM1} any $\sigma_0\in\partial \mathcal{C}$ yields a semiample class $\ell_{\sigma_0}$ on $M_\mathcal{C}(v)$, and the associated morphism $\pi:M_{\mathcal{C}}(v)\to M$ is a \Kthreen-type contraction.  Conversely, by \cite[Theorem 1.2(c)]{BM} any contraction arises from this construction.  By \cite[Theorem 1.1]{BM1} the morphism $\pi$ contracts a curve if and only if two generic stable objects in the corresponding family are $S$-equivalent with respect to $\sigma_0$.
\end{example}
Much is known about the singularities of Bridgeland contractions, and Theorem \ref{theorem huybrechts} roughly says that arbitrary \Kthreen-type contractions exhibit no new singularities:
\begin{proposition}\label{proposition def to bridge}Any \Kthreen-type contraction is locally trivially deformation-equivalent to a Bridgeland contraction.  Furthermore, if $b_2(X)>4$, it is locally trivially deformation-equivalent to an untwisted Bridgeland contraction.
\end{proposition}
\begin{proof}  Fix an identification $\Lambda_{K3^{[n]}}=H^2(Y,\Z)$, and let $\Lambda=\pi^*H^2(X,\Z)$.  If we choose an arbitrary primitive embedding $U\subset\tilde\Lambda_{K3}$ containing $\Lambda_{K3^{[n]}}^\perp$, then $U\cap\Lambda$ is of rank at most one and negative definite.  Assuming $b_2(X)>4$, then there is a rational-rank-zero projective period $\omega\in U^\perp\cap \Omega_\Lambda$, and we can find a point of its orbit under $\Gamma\subset\O(\Lambda)$ arbitrarily close to the period of $X$ by Proposition \ref{proposition verbitsky}.  Here we take $\Gamma$ to be the finite-index subgroup of isometries extending to $\tilde\Lambda_{K3}$ and stabilizing $\Lambda_{K3^{[n]}}$ (and $\Lambda$).  Thus, $\pi$ has a (small) locally trivial deformation to a contraction $\pi':Y'\to X'$ where $Y'$ is an untwisted Bridgeland moduli space, by Proposition \ref{proposition criterion}.  $\pi'$ is necessarily a Bridgeland contraction since every contraction of a Bridgeland moduli space is a Bridgeland contraction.  

If $b_2(X)=3$ or $4$, let $\Pi $ be the orthogonal to $\Lambda$ in $\tilde\Lambda_{K3}$.  As $\rk(\Pi)>4$, we can find $U\subset \Pi_\Q$ since $\Pi$ has (many) isotropic vectors by a classical theorem of Meyer, and the claim again follows from Corollary \ref{corollary projective new}.
\end{proof}
The work of Bayer-Macr\`i \cite{BM} in principle provides a complete description of the singularities of Bridgeland contractions as they are all realized by wall-crossing.  For the following we refer specifically to \cite[\S 5]{BM}.  A Bridgeland stability condition $\sigma_0$ on a projective twisted \Kthree surface $(S,\alpha)$ comes with a central charge $Z_0:\tilde H(S,\alpha,\Z)\to \C$, and we denote by $\hyplat_{\sigma_0}(v)\subset \tilde H(S,\alpha,\Z)_{\mathrm{alg}}$ the primitive sublattice of Hodge vectors $a$ such that $\Im \frac{Z(a)}{Z(v)}=0$, \emph{i.e.}, those vectors for which $Z(a)$ and $Z(v)$ are $\R$-linearly dependent.  For a generic $\sigma_0$ on a wall of the Bridgeland stability space, $\hyplat_{\sigma_0}(v)$ is of signature $(1,1)$ and 
\[
N_{\sigma_0}(v):= \hyplat_{\sigma_0}(v)\cap H^2(M_{\sigma}(v),\Z)=\hyplat_{\sigma_0}(v)\cap v^\perp
\]
 is generated by a vector of negative square.  A nearby generic stability condition $\sigma$ yields a relative Picard rank one contraction $\pi:M_\sigma(v)\to M$ which identifies $\sigma$-stable sheaves which are $S$-equivalent with respect to $\sigma_0$.  Denoting by $R_{\sigma_0}(v)\subset H_2(M_\sigma(v),\Z)$ the primitive sublattice corresponding to $N_{\sigma_0}(v)$ under the isomorphism $H^2(M_\sigma(v),\Q)\cong H_2(M_\sigma(v),\Q)$ given by the Mukai pairing, we see that $R_{\sigma_0}(v)$ is identified with $N_1(M_\sigma(v)/M)$, and $N_{\sigma_0}(v)_ \Q$ with the orthogonal to $\pi^*H^2(M,\Q)$ in $H^2(M_{\sigma}(v),\Q)$ (which we have called simply $N_\Q$ above).

For an arbitrary \Kthreen-type contraction $\pi:Y\to X$, we likewise denote by $N$ and $\mathcal{H}$ the orthogonals to $\pi^*H^2(X,\Z)$ in $H^2(Y,\Z)$ and $\tilde\Lambda(Y,\Z)$, respectively.  Relative Picard rank one contractions are particularly easy to analyze:
\begin{lemma}  \label{cor monodromy class}Let $\pi:Y\to X$ be a relative Picard rank one \Kthreen-type contraction.  Then the locally trivial deformation type of $X$ is uniquely determined by each of the following:
\begin{enumerate}
\item the abstract isomorphism class of $(\mathcal{H},v)$ as a pointed lattice;
\item $q_Y(\lambda)$ and $\div(\lambda):=[\Z:q_Y(\lambda,H^2(Y,\Z))]$, where $\lambda\in N$ is a primitive generator.
\end{enumerate}  
\end{lemma}
\begin{proof}  
Suppose $\pi':Y'\to X'$ is a second \Kthreen-type contraction with associated lattices $N'$ and $\mathcal{H}'$.  By {Corollary \ref{cor parallel transport},} $X$ and $X'$ are locally trivially deformation equivalent if and only if there is a parallel transport operator $f:H^2(Y,\Z)\to H^2(Y',\Z)$ sending $N$ to $N'$, which in this case is equivalent to there being an isometry $\tilde f:\tilde\Lambda(Y,\Z)\to\tilde\Lambda(Y',\Z)$ sending $(\mathcal{H},v)$ to $(\mathcal{H}',v')$.  Such a $\tilde f$ exists if and only if $(\mathcal{H},v)\cong (\mathcal{H}',v')$ by \cite[Corollary 1.5.2]{Ni}.  This proves (1).

For (2), $\Mon^2(K3^{[n]})$ orbits of primitive vectors $\lambda$ are uniquely determined by $\lambda^2$ and $\div(\lambda)$, see \cite[Section 10]{Eichler}.
\end{proof}

\begin{proposition}\label{nakajima}  Let $X$ be a relative Picard rank one \Kthreen-type contraction and $x\in X$ a (closed) point.  The analytic germ $(X,x)$ is isomorphic to that of a Nakajima quiver variety.
\end{proposition}
\begin{proof}By \cite[Theorem 1.1]{AS}, the statement is known for Gieseker moduli spaces $X=M_{H_0}(v_0)$ where $v_0$ is a primitive Mukai vector of a pure 1-dimensional sheaf on a $K3$ surface $S$ with $v_0^2\geq 2$ and $H_0$ is a nongeneric polarization.  Let $H$ be a generic polarization such that $M_H(v_0)\to M_{H_0}(v_0)$ is a symplectic resolution, and set $\mathcal{H}_0=\mathcal{H}(M_H(v_0)/M_{H_0}(v_0))$.  

By the previous lemma, we just need to show any pointed lattice $(\mathcal{H},v)$ arising from a relative Picard rank one contraction $\pi:Y\to X$ is abstractly isomorphic to some $(\mathcal{H}_0,v_0)$.  By Proposition \ref{proposition def to bridge} we may assume $\pi:Y\to X$ is a Bridgeland contraction, and then by \cite[Theorem 12.1]{BM} we must have that $\hyplat$ contains a class $a\in\hyplat$ with $0\leq(v,a)\leq (v,v)/2$ and $(a,a)\geq -2$.  Let $S$ be a $K3$ surface such that
\begin{enumerate}
\item $\Pic(S)\cong \hyplat$.  Let $D\in\Pic(S)$ correspond to $v$ and $A$ to $a$.
\item $D-\epsilon A$ is ample for $\epsilon>0$ sufficiently small.
\end{enumerate}
Such a surface $S$ exists since $\hyplat$ embeds primitively into $\Lambda_{K3}$ \cite[Theorem 1.1.2]{Ni}.  Then $D,A$ and $D-A$ are effective by the conditions on $a$, so $|D|$ contains reducible curves.  Choose an ample $H_0$ and $\delta,\alpha\in\Z$ nonzero such that
\[\frac{\delta}{H_0.D}=\frac{\alpha}{H_0.A}.\]
Then there are strictly $H_0$-semistable sheaves of Mukai vector $v_0=(0,D,\delta)$.  Indeed, taking an irreducible curve $C_1\in|A|$ (respectively $C_2\in|D-A|$) as well as a line bundle $L_1$ on $C_1$ (respectively $L_2$ on $C_2$) with Mukai vector $(0,A,\alpha)$ (respectively $(0, D-A,\delta-\alpha)$), any extension of $L_1$ by $L_2$ will be a strictly semistable sheaf with Mukai vector $v_0$.

Set $M_\pm=M_{H^\pm}(v_0)$ for $H^\pm=H_0\pm \epsilon D$, and $M_0=M_{H_0}(v_0)$.  We conclude by noticing that the lattice $\hyplat_0$ in $\tilde H(S,\Z)=\tilde\Lambda(M_+,\Z)$ associated to the wall crossing 
\[\xymatrix{
M_+\ar[dr]\ar@{-->}[rr]&&M_-\ar[ld]\\
&M_0&
}\]
is the saturation of $\left\langle (0,D,\delta), (0,A,\alpha)\right\rangle$, which is isomorphic to $\mathcal{H}$ and the isomorphism takes $v_0$ to $v$.
\end{proof}

The proof of Proposition \ref{nakajima} gives explicit models for every relative Picard rank one \Kthreen-type contraction among compactified Jacobians of linear systems on $K3$ surfaces.   Knutsen, Lelli-Chiesa, and Mongardi \cite{KLCM} have also used compactified Jacobians to construct contractible ruled subvarieties of \Kthreen-type manifolds, and analyze the geometry more closely.  Models for such contractions have further been treated by Hassett--Tschinkel \cite{HT15}, where it is shown that every wall $\hyplat$ can be realized on the Hilbert scheme of points for a projective $K3$ surface of Picard rank one.

The Bayer--Macr\`i picture strongly suggests that the answer to the following question is affirmative:
\begin{question}Let $\pi:Y\to X$ be a relative Picard rank one \Kthreen-type contraction, and let $E\subset Y$ be an irreducible component of the exceptional locus.  Is the generic fiber of the map $E\to X$ isomorphic to $\P^{\codim E}$?

\end{question}
Indeed, for a Bridgeland moduli space $Y=M_{\sigma_+}(v)$ and a contraction induced by a wall-crossing, the Harder--Narasimhan filtration of a generic point $[F]\in E$ with respect to a generic nearby stability condition $\sigma_-$ on the other side of the wall is often of the form
\begin{equation}0\to A\to F\to B\to 0\label{HNfilt}\end{equation}
for $A,B$ $\sigma_0$-stable.  All such extensions are $\sigma_+$-stable, and this yields a $\P^k=\P\Ext^1(B,A)$ fiber that is contracted.  Moreover, setting $a=v(A)$ and $b=v(B)$,
\begin{align*}\dim E&=k+\dim M_{\sigma_0}^\st(a)+\dim M^\st_{\sigma_0}(b)\\
&=\left((a,b)-1\right)+\left(a^2+2\right)+\left(b^2+2\right)\\
&=\left(v^2+2\right)-k\end{align*}
so $k=\codim E$.  Thus, in this case, we are done if the Harder--Narasimhan filtration of the general point of $E$ has the form \eqref{HNfilt} for fixed\footnote{Note the constancy of the Mukai vectors is automatic if a universal family exists over $E$, by the existence of Harder--Narasimhan filtrations in families.} $a$ and $b$.  By Lemma \ref{cor monodromy class} and \cite{LP2}, it would be sufficient to consider one model in each monodromy orbit, and many special cases have been established previously, see \cite{HT15,KLCM}.  

It is also not difficult to prove the following special case, since for ADE singularities the singularity determines the general fiber of the exceptional divisor of the (unique) symplectic resolution:

\begin{proposition}\label{proposition generic fiber}Let $\pi:Y\to X$ be a relative Picard rank one \Kthreen-type contraction, and let $E\subset Y$ be an irreducible divisorial component of the exceptional locus.  Then the exceptional locus is irreducible and the generic fiber of the map $E\to X$ is $\P^1$.
\end{proposition}
\begin{proof}
First suppose there is another irreducible component $E'$ of the exceptional locus, and let $C$ (resp. $C'$) be an integral curve in the general fiber of $E\to X$ (resp. $E'\to X$).  On the one hand we must have $E.C'\geq 0$, but on the other hand we must have $\widetilde{q}([E])\in \R_{>0}[C]$ for instance by deforming to a contraction $\pi':Y'\to X'$ for which $X'$ is of Picard rank zero and using both $q([E],[E])<0$ and $E.C<0$.  Thus, $[C']$ cannot be a positive multiple of $[C]$, a contradiction.   

In the notation of the proof of Proposition \ref{nakajima}, we may assume the contraction is of the form $\pi: M_H(v_0)\to M_{H_0}(v_0)$ where $v_0=(0,D,\delta)$.  $M_H(v_0)$ is stratified by the decomposition of $D$ into the supports of the Harder--Narasimhan factors, with the generic stratum corresponding to the trivial decomposition $D=D$ (even if $D$ has a fixed part).  The generic point of $E$ corresponds to a decomposition $D=D_1+D_2$; let $F$ be the generic fiber of $E\to X$.  For each component $F_0$ of $F$, any two generic points of $F_0$ have the same $H_0$-Harder--Narasimhan factors, which are sheaves $A_1,A_2$ supported on $D_1,D_2$ respectively.  There is a $\P^n$ of such sheaves---namely, the extensions $\P\Ext^1(A_1,A_2)$ (or in the reverse direction), all of which are semistable---which must be a $\P^1$ as its necessarily a curve.  Thus, all points in $F$ must be extensions of this form, and so $F\cong \P^1$. 
\end{proof}

In \cite{BB} the first author classified contractions at the other extreme, namely those for which the exceptional locus contains a Lagrangian $\P^n$.

As an application, assume that $\tX$ is a symplectic $2n$-fold which admits a divisorial contraction $\pi:\tX\to X$ of relative Picard rank one such that $X$ has transversal $A_2$ singularities. We call this an $A_2$-contraction. Note that while $ADE$ singularities admit unique symplectic resolutions, in the relative Picard rank one setting the monodromy action on the set of components of the general fiber of the exceptional locus yields a group of automorphisms of the $ADE$ graph in question acting transitively on the nodes, so only $A_1$ and $A_2$ singularities remain as possibilities.  We would like to know whether an $A_2$ contraction exists if $\tX$ is an irreducible symplectic manifold.

\begin{corollary}\label{corollary a2}
Let $\tX$ be a \Kthreen-type manifold. Then $\tX$ does not admit any $A_2$-contraction of relative Picard rank one.
\end{corollary}

Note that there are however examples of $A_2$-contractions of relative Picard rank one of smooth and projective symplectic varieties. See e.g. \cite[\S 1.4, Example 2]{Wierzba} for an explicit construction.


\end{document}